\documentclass[12pt]{article}
\usepackage{amsmath, amssymb, amsfonts, amsthm}
\newcommand{\R}{\mathbb R}

\newcommand{\C}{\mathbb C}

\newcommand{\Nb}{\mathbf N}
\newcommand{\nb}{{\mathbf n}}
\newcommand{\Fb}{\mathbf{F}}

\newcommand{\be}{\begin{equation}}
\newcommand{\ee}{\end{equation}}
\newcommand{\ddiv}{{\rm div\,}}
\newcommand{\tchi}{\widetilde{\chi}}

\newcommand{\bb}{\mathfrak{b}}

\newcommand{\Ab}{\mathbf{A}}
\newcommand{\Jac}{\rm Jac}
\def\0{\mathbf 0}
\def\q{\quad}
\def\qq{\qquad}
\def\qqq{\qq\qq}

\def\p{\partial}
\def\O{\Omega}
\def\k{\kappa}

\newtheorem{theorem}{Theorem}[section]
\newtheorem{prop}[theorem]{Proposition}

\newtheorem{corollary}[theorem]{Corollary}
\newtheorem{lemma}[theorem]{Lemma}
\newtheorem{definition}[theorem]{Definition}
\newtheorem{assumption}[theorem]{Assumption}

\theoremstyle{remark}
\newtheorem{remark}[theorem]{Remark}
\numberwithin{equation}{section}
\DeclareMathOperator{\gseg}{\mathrm{E^{g.st}_{\rm glob}}(\kappa,{\it q},\tau,{\it K}_1,{\it K}_2,{\it K}_3,{\it K}_4)}
\DeclareMathOperator{\gse}{\mathrm{E^{g.st}}(\kappa,{\it q},\tau,{\it K}_1,{\it K}_2,{\it K}_3)}
\DeclareMathOperator{\gsej}{\mathrm{E^{g.st}}(\kappa_j,{\it q}_{\it j},\tau,{\it K}_{1,\it j},{\it K}_{2,\it j},{\it K}_{3,\it j})}

\DeclareMathOperator{\gsejs}{\mathrm{E^{g.st}}(\kappa_{\it j_s},{\it q}_{\it j_s},\tau,{\it K}_1,{\it K}_2,{\it K}_3)}

\def\Eground{{\mathfrak e}_0}

\DeclareMathOperator{\supp}{supp}
\DeclareMathOperator{\Div}{div}
\DeclareMathOperator{\curl}{curl}

\DeclareMathOperator{\dist}{dist}
\newcommand{\eq}{\begin{equation}}
\newcommand{\eeq}{\end{equation}}
\usepackage[lmargin=2.5 cm,rmargin=2.5 cm,tmargin=3.5cm,bmargin=2.5cm,paper=a4paper]{geometry}

\title{Existence of Surface Smectic States of Liquid Crystals}

\author{S\o ren Fournais \thanks{Work partially supported by a Sapere Aude grant from the Danish Councils for Independent Research, Grant number DFF--4181-00221; {\text fournais@math.au.dk}} \\ Department of Mathematics, Aarhus University\\
 Ny Munkegade 118, DK-8000 Aarhus, Denmark\\
 and \\
Ayman Kachmar\thanks{\text{Work partially supported by the Lebanese University Research Funds; ayman.kashmar@gmail.com}} \\
Lebanese University, Department of Mathematics, Hadath, Lebanon.\\
and \\
Xing-Bin Pan \thanks{Work partially supported
by the National Natural Science Foundation of China grant no. 11171111 and no. 11431005, and the Chinese Specialized Research Fund for the Doctoral Program of Higher Education grant no. 20110076110001; {\text xbpan@math.ecnu.edu.cn}}\\ Department of Mathematics, East China Normal University, and\\ NYU-ECNU Institute of Mathematical Sciences at NYU Shanghai,  \\ Shanghai 200062, People's Republic of China}

\begin{document}

\date{\today}
\maketitle



\begin{abstract}
The  Landau-de Gennes  model of  liquid crystals is a functional
acting on wave functions (order parameters) and vector fields (director fields).
In a specific asymptotic limit of the physical
parameters, we construct critical points such that the wave function (order
parameter) is localized near the boundary of the domain, and we determine
a sharp localization of the boundary region where the
wave function concentrates. Furthermore, we compute the asymptotics
of the energy of such critical points along with a boundary energy that may serve in localizing the director field. In physical terms, our results prove the
existence of a surface smectic state.
\end{abstract}

\tableofcontents

\section{Introduction}\label{sec:int}
\subsection{The model}
In this paper, we  investigate special critical points of  the
Landau-de\,Gennes energy functional. These critical points are constructed by minimization over a special class of configurations.

Critical/minimizing  configurations of the Landau-de\,Gennes  functional  were
introduced by De\,Gennes in \cite{deGe} to model the
nematic-smectic\,$A^*$ phase transition in a liquid crystal near the
critical temperature.
The Landau-de\,Gennes energy functional acts on configurations of
the type $(\psi,\nb)$, where $\psi$ is a complex-valued function and
$\nb$ is a vector field. Both $\psi$ and $\nb$ are defined in a
three dimensional domain $\Omega$, assumed to be the region occupied
by the molecules of the liquid crystal. The complex-valued  function
$\psi$,  called the order parameter, detects the
nematic/smectic phase as follows: $\psi(x)=0$ indicates a (chiral)
nematic phase near $x$, while $\psi(x)\not=0$ indicates a (chiral)
smectic\,$A^*$ phase. The length of the vector field $\nb$, also called
the director field, is constrained such that $|\nb|=1$. The
direction of the director field  has a physical meaning: $\nb(x)$ is
the average direction of the liquid crystal molecules in a small region around $x$.

The Landau-de\,Gennes energy is given by
\begin{equation}\label{eq-LdeG}
{\mathcal E}(\psi,\nb) =\mathcal G(\psi,\nb) +\mathcal F_N(\nb)\,,
\end{equation}
where
\begin{align}\label{eq:GL}
\mathcal G(\psi,\nb)&:= \int_{\Omega} |\nabla_{q \nb} \psi|^2 - \kappa^2|\psi|^2 + \frac{\kappa^2}{2} |\psi|^4\,dx\,,
 \end{align}
with the notation $\nabla_{q\nb}:=\nabla-iq\nb$, and
\begin{align}\label{eq-OFE}
\mathcal F_N(\nb):=&{\mathcal F}_N^{+}(\nb) + {\mathcal L}(\nb)\,,
\end{align}
with
\begin{align}\label{eq:OFPE}
\mathcal F_{N}^+(\nb)&:=\int_{\Omega}\Big\{ K_1 |\Div \nb |^2 + K_2 |(\curl \nb )\cdot \nb + \tau |^2 + K_3 |\curl \nb \times \nb |^2\Big\}\,dx\,,\\
\label{eq:OFNL}
\mathcal L(\nb)&:=(K_2+K_4)\int_\Omega\Big({\rm tr}(D\nb)^2-|{\rm div}\nb|^2\Big)\,dx\,.
\end{align}
The first term
$\mathcal G(\psi,\nb)$ is reminiscent of the energy for the order
parameter in the Ginzburg-Landau theory of superconductivity. The
term $\mathcal F_N(\nb)$ is the Oseen-Frank energy of the director
field and consists of a non-negative part $\mathcal F_{N}^+(\nb)$
and a part $\mathcal L(\nb)$ which is ({\it a priori}) with  indefinite sign.
The behavior of the minimizers of the functional $\mathcal F_{N}^+$
is well known, this will be recalled in a detailed review below.
Moreover these minimizers correspond to the zeros of $\mathcal
F_{N}^+$  if no further boundary constraint is imposed on the vector field $\nb$. The  functional $\mathcal L(\nb)$ is a null Lagrangian,
that is, the  value of $\mathcal L(\nb)$ is  determined  by the value of $\nb$ on the boundary.
 In fact if $\nb$ is sufficiently regular, then we have
$$
\int_\Omega\Big({\rm tr}(D\nb)^2-|{\rm div}\nb|^2\Big)\,dx= \int_{\Omega} \Div\big( (\nabla\nb) \nb - \Div(\nb) \nb \big)\,dx\,.
$$
If we minimize the Landau-de\,Gennes functional
under a fixed Dirichlet boundary condition on the director field, then the
integral of the null Lagrangian will be constant and will not play any role.
See \cite{HKL, L1, L2, LL3} for the above mentioned facts and more on the mathematical theory of static nematic liquid crystal configurations in the Oseen-Frank theory, \cite{LL1, LL2, LL3, LW} and the references therein for recent progress on the mathematical theory of liquid crystals flows,  \cite{C, BCLP, P2, JP} and the references therein for the mathematical analysis on the Landau-de Gennes model \eqref{eq-LdeG}.

In the expression of the Landau-de\,Gennes functional, various
parameters appear. The parameter $\kappa>0$ is material
dependent and depends on the temperature. That $\kappa>0$
signifies that the temperature is below the critical temperature
responsible for the nematic-smectic transition for a non-chiral
material. By analogy with superconductivity, we may call $\kappa$
the Ginzburg-Landau parameter, and we will consider the regime of
large $\kappa$.

The two other parameters left in $\mathcal E(\psi,\nb)$ are $q$ and
$\tau$. Here
$\tau$ is the chiral twist ($\tau=0$ indicates a non-chiral
material). We will suppose that $\tau>0$. The parameter $q>0$ measures the density
of smectic layering.

The positive constants $K_i$ in the functional $\mathcal F_N$,
$i=1,2,3$, correspond to the splay, twist and bend elasticity
constants.
The
constant $K_4$ is usually selected in $(-\infty,0]$, see \cite{dGP, E},
while $K_1,K_2,K_3$ are typically large. We refer to \cite{dGP} for
more details regarding the physics behind the functional in
\eqref{eq-LdeG}.

In \cite{deGe}, De\,Gennes pointed out an analogy  between liquid
crystals and superconductivity. Guided by this analogy,  one might
expect to find a {\it surface smectic state} for certain values of
the various parameters, exactly as one finds a surface
superconducting state for type I\!I superconductors.  Such a state
corresponds to a minimizing/critical  configuration $(\psi,\nb)$ such that
$\psi$ is concentrated in a narrow region around the boundary
$\partial \Omega$ of the sample, see \cite[p.346]{P2} and \cite[Problem 3.2.4]{P4}.

This question has been addressed in many papers, see \cite{A, HP08, P2}
(and the references therein) using techniques developed for the
Ginzburg-Landau functional. Loosely speaking, the obtained results
for the Landau-de\,Gennes functional correspond to similar results
in superconductivity by understanding the role of $q\tau$  as being
the same role played by the intensity of the applied magnetic field in superconductivity
(recall that the quantity $q\tau$ involves the chirality constant
$\tau$ and the number of smectic layers $q$).

The mathematics behind this is not easy. Using the known techniques from the study of the
Ginzburg-Landau functional, we may estimate the energy $\mathcal G(\psi,\nb)$
if we know some {\it a priori} estimates satisfied by the vector
field $\nb$. In the context of the Ginzburg-Landau functional (as
opposed to the Landau-de\,Gennes functional ${\mathcal E}(\psi,\nb)$)
such {\it a priori} estimates hold by using  PDE techniques, namely the div-curl
regularity theorems. This is true since the `magnetic energy' in the Ginzberg-Landau functional is elliptic/coercive, unlike the complicated energy
$\mathcal F_N$ in \eqref{eq-OFE}.
To break
this circle, Helffer-Pan \cite{HP08} dropped the term with the
indefinite sign in $\mathcal F_N$. That can be done either by taking
$K_4=-K_2$ or by imposing a Dirichlet condition on the director
field so that the null Lagrangian term is constant. The minimizers
of the Landau-de\,Gennes functional are then analyzed through two
successive limits, first a reduced functional is obtained by passing
to the limit $\min(K_1,K_2,K_3)\to\infty$, then it is proved that
the minimizing order parameter is localized in a thin boundary
region by passing to the limit $q\approx \kappa^2$ and
$\kappa\to\infty$. Among the results in this paper,  we derive
additional information about the localization of the minimizing
order parameter by passing through one single limit. To do this, we
do not  drop the null Lagrangian term in order to extract useful
estimates about the director field, but we impose a Dirichlet boundary condition. Also, we justify imposing such a Dirichlet boundary condition by proving that it is asymptotically true  in a special limit.

The additional  estimates we derive
allow us to estimate the Ginzburg-Landau energy $\mathcal
G(\psi,\nb)$ and the concentration of the order parameters using the
approach developed recently in
\cite{FKP3D} for the 3D Ginzburg-Landau functional.

\subsection{Limiting Dirichlet condition for the director field}

Ideally, one would like to minimize  the energy in \eqref{eq-LdeG} without imposing boundary conditions on the configurations $\psi$ and $\nb$. This leads to the introduction of the following ground state energy,
\begin{equation}\label{eq-gs-LdeG*}
\gseg=\inf\,\big\{\mathcal E(\psi,\nb)~:~(\psi,\nb)\in H^1(\Omega;\C)\times H^1(\Omega;\mathbb S^2)\big\}\,.
\end{equation}
Notice that the admissible configurations $\nb$ are constrained,
$|\nb|=1$. This makes the study of the minimizers more complex. For
instance, the Euler-Lagrange equation for the director field is
quite complicated to handle by PDE techniques. The reader is referred to \cite{HKL, L1, L2, LL3}
and the references therein for the mathematical analysis on this model, also
\cite[Sec.~4]{R} for a discussion of the Euler-Lagrange equation for the director fields.

In the mathematical theory of liquid crystals, Dirichlet boundary conditions for $\nb$ have been used very often, e.g. \cite{HKL, L1, L2, L3, LL1, LL2, LL3, LP, LW, P2, R}.  In the physics literature, a boundary condition on the director field models the interaction between the liquid crystal molecules and the molecules near the wall of the container. When the liquid crystal molecules attract the molecules of the container, the director field satisfies the boundary condition  $\nb\cdot\Nb=0$. When the molecules repel each other, the boundary condition becomes $\nb\times\Nb=\0$. Here $\Nb$ is the unit outward normal vector on the boundary.

In  this paper, we will impose a Dirichlet boundary condition on the director field. We will justify imposing such a  Dirichlet condition in a specific asymptotic limit (see Theorem~\ref{thm:HP}). At the same time, the minimizers obtained under such a Dirichlet condition are still critical points of the functional in \eqref{eq-LdeG}.

We will exhibit an asymptotic limit where we can determine the boundary behavior of the minimizing director field, thereby deriving a limiting form of a Dirichlet boundary condition, that we will impose later in this paper.

The starting point is to find minimizers of the non-negative part of
the Oseen-Frank energy, i.e. the functional $\mathcal F_N^+$ in
\eqref{eq:OFPE}. Let us introduce the following set,
\begin{equation}\label{eq:C(tau)}
\mathcal C(\tau)=\big\{\,\nb\in H^1(\Omega;\mathbb S^2)~:~\mathcal F_N^+(\nb)=\min_{\mathbf u\in H^1(\Omega;\mathbb S^2)}\mathcal F_{N}^+(\mathbf u)\,\big\}\,.
\end{equation}
In \cite{BCLP}, it is proved that:
\begin{align}\label{C(tau)}
\mathcal C(\tau) &= \{ \nb \,:\, \Div\nb=0 \text{ and }
\curl\nb+\tau\nb=\0 \text{ in }\Omega\} \\
\label{eq:C(tau)rotated}
&=\big\{\, N_{\tau}^Q (\cdot ):=Q N_{\tau}(Q^t \,\cdot) ~:~ Q \in SO(3)\,\big\}\,,
\end{align}
where
\begin{align}\label{Ntau}
N_{\tau}(x) = ( \cos(\tau x_3), \sin(\tau x_3), 0).
\end{align}

Through \eqref{eq:C(tau)rotated} the set ${\mathcal C}(\tau)$ can be identified with $SO(3)$, i.e. ${\mathcal C}(\tau)$ is naturally a compact metric space. Hence, continuous functions defined in
$\mathcal C(\tau)$ have maxima and minima in $\mathcal C(\tau)$.

The analysis in this paper is valid when the following two assumptions are satisfied:

\begin{assumption}\label{assumption:A}~
\begin{enumerate}
\item $\Omega\subset\R^3$ is an open domain,
$\partial\Omega$ is $C^{2,\alpha}$ smooth with $0<\alpha<1$,
compact and consists of a finite number
of connected components.
\item $M>1$ is a  constant\,;
\item  $e(\kappa)$ is a positive function satisfying $\displaystyle\lim_{\kappa\to\infty}\frac{e(\kappa)}{\kappa^{2}}=\infty$\,;
\item $(\kappa_j,q_j,\tau_j,K_{1,j},K_{2,j},K_{3,j})$ is a sequence in $\R_{+}^6$ such that $\displaystyle\lim_{j\to\infty}\kappa_j=\infty$ and
$$\forall~j\,,\quad0< e(\kappa_j)\leq\min(K_{1,j},K_{2,j},K_{3,j})\leq \max(K_{1,j},K_{2,j},K_{3,j})\leq  Me(\kappa_j)\,.$$
\end{enumerate}
\end{assumption}

The Assumption~\ref{assumption:A} is needed for Theorem~\ref{thm:HP} below. For the more detailed convergence results of Theorems~\ref{thm:en*}~and~\ref{thm:op} we will need the additional Assumption~\ref{assumption:A'}.

\begin{assumption}\label{assumption:A'}~
\begin{enumerate}
\item $b>1$ and $\tau>0$ are constants\,;
\item The function $e(\kappa)$ in Assumption~\ref{assumption:A} satisfies
$\displaystyle\lim_{\kappa\to\infty}\frac{e(\kappa)}{|\ln\kappa|^2\kappa^{3}}=\infty$\,;
\item The sequence in Assumption~\ref{assumption:A} satisfies, for all $j$,
$$\tau_j=\tau\quad{\rm and}\quad q_j\tau=b\kappa_j^2\,.$$
\end{enumerate}
\end{assumption}

The asymptotic behavior of the minimizing director fields is contained in:

\begin{theorem}\label{thm:HP}
Let $c_2>c_1>0$ be constants.
Suppose that Assumption~\ref{assumption:A} is satisfied and  $\{K_{4,j}\}$ is a sequence in $\R$ satisfying
$$\forall~j\,,\quad c_1\kappa_j^2\leq K_{2,j}+K_{4,j}\leq c_2\kappa_j^2\,.$$
Let $(\psi_j,\nb_j)$
be a minimizer of the functional in \eqref{eq-LdeG} for
$(\kappa,q,\tau,K_i)=(\kappa_j,q_j,\tau_j,K_{i,j})$.
Then there exist $\nb_0\in\mathcal C(\tau)$ and a subsequence  $\{(\psi_{j_s},\nb_{j_s})\}$ such that, as $s\to\infty$,
$$
\nb_{j_s}\to\nb_0,
$$
where the convergence is in $H^1_{\rm loc}(\Omega,\R^3) \cap L^p(\Omega;\R^3) \cap W^{1,r}(\Omega;\R^3)$, for all $p \in [1,\infty)$ and $r \in [1,2)$, and also in $L^{r}(\partial\Omega;\R^3)$, $1\leq r<2$.
\end{theorem}

Note that Theorem~\ref{thm:HP} adds to the conclusions in \cite{HP08} since we establish the convergence of $\nb_{j_s}$ along the boundary (precisely in the space $W^{1,r}(\Omega;\R^3)$). In \cite{HP08}, the convergence holds in $H^1_{\rm loc}(\Omega;\R^3)$ and $L^p(\Omega;\R^3)$ only.

\subsection{Critical points and boundary energy for the director
field}\label{sec:ren}

In light of Theorem~\ref{thm:HP}, we introduce the following admissible class for the
director fields,
\begin{equation}\label{eq:adm-n}
\mathcal A=\{\nb\in H^1(\Omega;\mathbb S^2)~:~\exists~ \nb_0\in\mathcal C(\tau)\;\;\text{\rm such that}\;\; \nb=\nb_0{\rm ~on~}\partial\Omega\}\,.
\end{equation}
If $\nb\in \mathcal A$, then
${\mathcal L}(\nb) = {\mathcal L}(\nb_0) = 0$ by \eqref{eq:C(tau)rotated}.
Thus, for every $(\psi,\nb)\in H^1(\Omega;\C)\times\mathcal A$, the expression of the functional in \eqref{eq-LdeG} simplifies to
\begin{equation}\label{eq-LdeG*}
\mathcal E(\psi,\nb)=\mathcal G(\psi,\nb)+\mathcal F_N^+(\nb)\,.
\end{equation}

Now, we introduce the ground state energy,
\begin{equation}\label{eq-gs-LdeG}
\gse=\inf\{\mathcal E(\psi,\nb)~:~(\psi,\nb)\in H^1(\Omega;\C)\times \mathcal A\}\,.
\end{equation}
In light of \eqref{eq-LdeG*}, the ground state energy $\gse$ is independent of the coefficient
$K_4$. A minimizer $(\psi,\nb)$ of \eqref{eq-gs-LdeG} is a critical point of the Landau-de\,Gennes functional in the following sense:
\begin{multline*}
\forall~(\phi,{\mathbf w})\in H^1(\Omega;\C)\times C_c^\infty(\Omega;\R^3)\,,\\
\frac{d}{dt}\mathcal E\left(\psi+t\phi,\nb\right)\Big|_{t=0}=0\quad {\rm ~and~}\quad
\frac{d}{dt}\mathcal E\left(\psi,\frac{\nb+t{\mathbf w}}{|\nb+t{\mathbf w}|}\right)\Big|_{t=0}=0\,.
\end{multline*}

Through the analysis in this paper,  we will describe the asymptotic behavior of the ground state energy in \eqref{eq-gs-LdeG} and its minimizers.  Also, we will derive a {\it
boundary} energy that might hint at the expected profile of
the director field of a minimizer.

\subsubsection{The boundary energy}

The definition of the boundary
energy involves many implicit quantities as well as the geometry of
the boundary of the domain $\Omega$.
The starting point is a local boundary energy function, which was constructed in \cite{FKP3D}. We will recall this construction in Section~\ref{sec:redGL} below. Here we just introduce the notation and some basic properties.

For $\bb \in (0,1]$ and $\nu \in [0,\pi/2]$, let $E(\bb,\nu)$ be the quantity defined by \eqref{eq:E-FKP3D} below. Then $E(\bb,\nu)$ is a continuous function of $(\bb,\nu)$. In the later use in the paper we have $\bb = 1/b$, with $b$ the constant from Assumption~\ref{assumption:A}.

Now we can give the definition of the boundary energy for the
director field.  Let $\nb\in\mathcal C(\tau)$. If
$x\in\partial\Omega$, denote by $\nu(x;\nb)$ the angle in
$[0,\pi/2]$ between $\curl\nb$ (which is equal to $-\tau\nb$) and the tangent plane to
$\partial\Omega$ at $x$.

Define,
\begin{align}
&\tilde E(\bb,\nb)=\int_{\partial\Omega}
E\big(\bb,\nu(x;\nb)\big)\,ds(x)\,,\label{E(bnu)}\\
&\Eground(\mathfrak{b},\tau)=\inf_{\nb\in\mathcal
C(\tau)}\tilde E (\bb,\nb)\,,
\label{eq-F0(nb)}
\end{align}
where $E(\bb,\nu)$ is defined in \eqref{eq:E-FKP3D}.
Clearly, a minimizer of $\tilde E(\bb,\nb)$ in $\mathcal C(\tau)$
exists (by continuity and compactness). We therefore introduce
\begin{equation}\label{eq:M(tau)}
\mathcal M=\mathcal M(\bb,\tau)=\big\{\,\nb\in \mathcal
C(\tau)~:~\tilde E(\bb,\nb)=\Eground(\bb,\tau)\,\big\}\,.
\end{equation}

The vector fields in $\mathcal M$ are determined by the boundary
geometry of the domain $\Omega$. If $\Omega$ is a ball, then it is
invariant by rotation and therefore $\mathcal M=\mathcal C(\tau)$.
For general domains, e.g. an ellipsoid, the set of minimizers
$\mathcal M$ is expected to be a proper subset of $\mathcal
C(\tau)$. To determine this set is an interesting question.
One would expect that `generically', ${\mathcal M}$
consists of a single vector field.

The boundary energy $\Eground(\mathfrak{b},\tau)$ might help to
determine the behavior of the minimizing director field. If we come back to the content of Theorem~\ref{thm:HP}, we find that  a sequence of minimizing director fields converges to a vector field $\nb_0\in\mathcal C(\tau)$. We will localize $\nb_0$ further by proving  that $\nb_0\in\mathcal M$.
However, the study of the minimizers
of the boundary energy seems complicated,  since this energy is
defined by implicit quantities.

\subsubsection{Concentration of the order parameter}

We present here one among the main results proved in this paper. It
adds over the results of Helffer-Pan \cite{HP08} a new formula for
the concentration of the minimizing order parameter valid by passing
through a single limit (in \cite{HP08}, the concentration is obtained
through  two successive limits).
The starting point is a leading order estimate of the ground state energy.

\begin{theorem}\label{thm:en*}
Let $\tau>0$ and $b>1$ be fixed constants and suppose that the conditions in  Assumptions~\ref{assumption:A} and \ref{assumption:A'} are satisfied. {As $\kappa_j\to\infty$,} the ground
state energy in \eqref{eq-gs-LdeG} satisfies,
$$\gsej=\sqrt{b}\, \kappa_j\,\Eground\Big(\frac1b,\tau\Big) +o(\kappa_j)
\,,$$
where $\Eground$ is defined in \eqref{eq-F0(nb)}.
\end{theorem}

The conclusion in Theorem~\ref{thm:en*} shows that the boundary is dominant  in the description of $\gsej$. In Theorem~\ref{thm:op} below, we will be more specific regarding the localization near the boundary. In particular, previous results (e.g.~\cite{A}) only proved decay of $\psi$ outside a certain `allowed' boundary region. We provide confirmation that $\psi$ is indeed not small in this boundary region.

The statement of Theorem~\ref{thm:op} below  involves a certain class of
sub-domains in $\Omega$  that we will call {\it regular} domains.

\begin{definition}\label{def:dom}
Let $D\subset\Omega$. We say that $D$ is {\bf regular} if
\begin{enumerate}
\item $D=\widetilde D\cap\Omega$ where $\widetilde D\subset\R^3$ is
open and has a smooth boundary\,;
\item If $\overline{D}\cap\partial\Omega\neq\emptyset$, then it consists of a finite number
of connected components\,;
\item Every connected component of $\overline{D}\cap\partial\Omega$
is either
\begin{enumerate}
\item a  smooth surface
 without boundary\,; or
\item
 a smooth surface with boundary consisting of a finite number of disjoint smooth curves\,;
\end{enumerate}
\item $\partial\widetilde D$ intersect $\partial\Omega$ transversally, i.e. the unit normal vectors $\Nb_{\partial\widetilde D}$ and $\Nb_{\partial\Omega}$ satisfy
$$\Nb_{\partial\widetilde D}\times\Nb_{\partial\Omega}\not=0\,.
$$
\end{enumerate}
\end{definition}

The assumptions we made on the  domain $\Omega$ assert that it is a regular domain. The exterior (in $\Omega$) of a regular domain is a regular domain too (i.e. if $D$ is regular, then $\Omega\setminus\overline{D}$ is regular). If $B$ is an open ball in $\R^3$ such that $\overline{B}\subset \Omega$ then $B$ is a regular domain.  If $\p\O$ is smooth, and if the center of $B$ lies on $\partial\Omega$ and the radius of $B$ is sufficiently small, then the domain $B\cap\Omega$ is
another example of a regular domain.
Another example of a regular domain
is the tubular neighborhood of the boundary of $\Omega$, i.e.
$D=\{x\in\Omega~:~{\rm dist}(x,\partial\Omega)< t\}$ where $t>0$ is
sufficiently small.

The results here are  valid under the hypotheses in Assumption~\ref{assumption:A'}.
We shall  examine the behavior of the minimizers of \eqref{eq-gs-LdeG} for large $\kappa$, hence $q\tau=b\kappa^2$ increases, and the elastic coefficients $K_1, K_2, K_3$ behave as $e(\kappa)\gg \kappa^3|\ln\kappa|^2$ by the conditions above.

\begin{theorem}\label{thm:op}
Under the conditions of Assumptions~\ref{assumption:A} and \ref{assumption:A'},
let $(\psi_j,\nb_j)$ be a minimizer of the energy in
\eqref{eq-gs-LdeG} for $(\kappa,q,\tau,K_i)=(\kappa_j,q_j,\tau,K_{i,j})$.
Then there exists a subsequence $\{(\psi_{j_s},\nb_{j_s})\}$ and $\nb_0 \in {\mathcal M}$ such that,  as $s\to\infty$:
\begin{enumerate}
\item $\nb_{j_s}\to\nb_0$ in $H^1_{\rm loc}(\Omega,\R^3)$ and in every $W^{1,r}(\Omega,\R^3)$ and $L^p(\Omega,\R^3)$, $1\leq r<2$ and $2\leq p<\infty$\,;
\item The Oseen-Frank energy in  \eqref{eq:OFPE} satisfies,
$$\mathcal F_N^+(\nb_{j_s})=o(\kappa_{j_s})\,.$$
\item If $D\subset \Omega$ is a {\bf regular} domain, then
\begin{multline}\label{eq:locEnergy}
\int_{D}\Big\{|(\nabla-iq_{j_s}\tau\nb_{j_s})\psi_{j_s}|^2-\kappa_{j_s}^2|\psi_{j_s}|^2
+\frac{\kappa_{j_s}^2}2|\psi_{j_s}|^4\Big\}\,dx \\
=
\sqrt{b} \,\kappa_{j_s}\,\int_{\overline{D}\cap\partial\Omega}
E\Big(\frac{1}{b},\nu(x,\nb_0)\Big)\,ds(x)
+ o(\kappa_{j_s}),
\end{multline}
where $E(\cdot,\cdot)$ is defined in \eqref{eq:E-FKP3D} and $\nu(x,\nb_{j_s})$ is the angle in $[0,\pi/2]$ between $\nb_{j_s}$ and the tangent plane to $\partial\Omega$ at $x$.
\item The subsequence $\{\psi_{j_s}\}$ has the following concentration behavior
$$
\kappa_{j_s}|\psi_{j_s}|^4\,dx\to -2\sqrt{b}\,E\Big(\frac1b,\nu(x;\nb_0)\Big)\,ds(x),
$$
which holds in the following sense: If $D\subset \Omega$ is a regular domain, then
\eq\label{L4}
\kappa_{j_s}\int_D|\psi_{j_s}|^4\,dx\to
-2\sqrt{b}\,\int_{\overline{D}\cap\partial\Omega}E\Big(\frac1b,\nu(x;\nb_0)\Big)\,ds(x)\,,
\eeq
where $dx$ is the Lebesgue measure in $\Omega$, $ds(x)$ is the
surface measure in $\partial\Omega$, and $E(\cdot,\cdot)$ is
defined in \eqref{eq:E-FKP3D}.
\end{enumerate}
\end{theorem}

Notice that if ${\mathcal M}=\{\nb_0\}$, i.e. if $\mathcal M$ consists
of a single field, then there is no need to extract a subsequence.

\begin{remark}
Theorem~\ref{thm:op}  indicates
that as long as Assumptions~\ref{assumption:A} and \ref{assumption:A'} are satisfied, then as
$\kappa\to\infty$, the order parameter concentrates at the boundary.
More precisely, using Remark~\ref{rem:useful} below where the spectral function $\zeta(\nu)$ is introduced, we see that since $E(\frac{1}{b},\nu)$ vanishes when $\zeta(\nu)\geq \frac{1}{b}$, we have the following conclusion:
\begin{itemize}
\item[(i)] If $b\geq \frac{1}{\Theta_0}$, then for any $x\in \p\O$ we have $\zeta(\nu(x;\nb_0))\geq \Theta_0\geq \frac{1}{b}$, hence
$$E\Big(\frac{1}{b},\nu(x,\nb_0)\Big)=0\q\text{\rm for any }x\in\p\O.
$$
Therefore, from \eqref{L4} we see that the order parameters $\psi$ are uniformly small in the sense that for any regular subdomain $D$
$$
\int_D|\psi|^4\,dx=o(\kappa^{-1}).
$$
This suggests that when $\kappa$ is large and $q\tau/\kappa^2$ is above the critical value $1/\Theta_0$ and kept away from it, the whole liquid crystal sample is in the nematic state.

\item[(ii)] If $1<b<\frac{1}{\Theta_0}$, then
$$E\Big(\frac{1}{b},\nu(x,\nb_0)\Big)<0\q\text{\rm if } \zeta\big(\nu(x;\nb_0)\big)<\frac1b.
$$
From Theorem \ref{thm:op}, for $q\tau=b\kappa^2$, the order parameter $\psi$ is localized near
the following region:
\begin{equation}\label{eq:bndlayer}
\mathcal
S_b(\nb_0)=\Big\{\,x\in\partial\Omega~:~\zeta\big(\nu(x;\nb_0)\big)<\frac1b\,\Big\}\,,
\end{equation}
which is a proper subset of $\p\O$. This suggests that when $\kappa$ is large and $q\tau/\kappa^2$ lies between but away from $1$ and $1/\Theta_0$, the liquid crystal is in the surface smectic state, and the surface portion near $\mathcal S_b(\nb_0)$ is in the smectic state, and all  the other part of the sample is in the nematic state.

\item[(iii)] If $b$ decreases and reaches $1$, then $\mathcal S_b(\nb_0)$ expands and covers the whole boundary, thus the smectic layer expands and eventually covers the whole surface.

\item[(iv)]   We expect that when $q\tau= b\kappa^2$ with $0<b<1$ then the whole sample is in the smectic state.
\end{itemize}
\end{remark}

The above remark indicates the analogy between the surface superconducting state of type I\!I superconductors (see \cite[Theorem 1 and Remark 1.2]{P} and \cite[Theorem 2 and Remark 1.2]{Pa}) and the surface smectic state of liquid crystals.  Also, the claim in (iv) is reminiscent of  bulk superconductivity in 3D samples \cite{FK-cpde}.

In \cite{HP08}, it is proved that the order parameter decays
exponentially fast away from the boundary region $\mathcal
S_b(\nb_0)$ in \eqref{eq:bndlayer}, but (unlike the formula in \eqref{L4}) it is not proved that the order parameter is not small everywhere in $\mathcal S_b(\nb_0)$. The result in \cite{HP08} is valid in the
framework of two successive limits. What we prove here is valid in
a single limit and yields that the concentration set of the order
parameter is {\it exactly} the region $\mathcal S_b(\nb_0)$.

\subsection{Behavior  of the ground state energy}

Here we indicate how our methods improve some of the results of Helffer-Pan \cite{HP08}. Also, we discuss  a possible method for localizing the director field
of a minimizing configuration.

We will analyze two
successive limits as done by Helffer-Pan in \cite{HP08}. The
analysis here is valid under the assumption that $K_4=-K_2$ in
\eqref{eq-OFE} so that the term with indefinite sign is dropped from
the Oseen-Frank energy. That is the assumption considered in
\cite{HP08}.

The results here will highlight the importance of the boundary
energy in \eqref{eq-F0(nb)} and are consistent with the contents of Theorems~\ref{thm:HP} and \ref{thm:op}. The domain $\Omega$ is again supposed
bounded, open and regular (see Definition~\ref{def:dom}).

\begin{theorem}\label{thm:gse}
Let $b>1$ and $\tau>0$ be fixed constants. Suppose that $K_4=-K_2$ and $q\tau=b\kappa^2$. The ground
state energy in \eqref{eq-gs-LdeG} satisfies,
\begin{equation}\label{eq-renormalized-energy}
\lim_{\kappa\to\infty}\left[\lim_{\min(K_1,K_2,K_3)\to\infty}\frac{\gse}{\sqrt{q\tau}}\right]=\Eground\left(\frac1b,\tau\right)\,,
\end{equation}
where $\Eground(\frac1b,\tau)$ is introduced in
\eqref{eq-F0(nb)}.
\end{theorem}

We stress again that in the limit \eqref{eq-renormalized-energy}, as $\kappa\to\infty$, $q\tau=b\kappa^2$ also goes to $\infty$.

The relevance of the calculation of the limit in
\eqref{eq-renormalized-energy} is more apparent in light of a
result by Helffer-Pan \cite{HP08} that we discuss below.

In \cite{HP08}, it is proved that, as $\min(K_1,K_2,K_3)\to\infty$, the
ground state energy
$$\gse$$
 converges to the ground state energy of a
reduced functional. The reduced functional is defined over
configurations in the space $H^1(\Omega;\mathbb C)\times\mathcal
C(\tau)$. Furthermore, if $(\psi,\nb)$ is a minimizer achieving
$\gse$, then as $\min(K_1,K_2,K_3)\to\infty$ along a subsequence, the
minimizer $(\psi,\nb)$ converges to $(\psi_0,\nb_0)$ where
$(\psi_0,\nb_0)$ minimizes the reduced functional. In the specific
regime
$$q\tau=b\,\kappa^2\,,\quad\Theta_0<\frac{1}{b}<1\,,\quad{\rm and~}\kappa\to\infty\,,
$$
it is proved that $\psi_0$ is localized near the boundary of
$\Omega$ and is exponentially small away from the boundary region
$\mathcal S_{b}(\nb_0)$ in \eqref{eq:bndlayer}. However,
the localization of the  field $\nb_0$ is left open. As a byproduct
of the proof of Theorem~\ref{thm:gse}, we get that $\nb_0$
is a minimizer of the boundary energy in \eqref{eq-F0(nb)}.

The results in this paper suggest the following question: Study the minimizers of $\tilde
E[\bb,\cdot]$ and study the geometric meaning of the minimizers.

\section{Behavior of the director field}
The aim of this section is to prove Theorem~\ref{thm:HP}. This is based on the result of the following lemma:
\begin{lemma}\label{lem:HP}
Let $C>0$ be a constant and let $r(\kappa):\R_+\to \R_+$ be a function satisfying $\lim_{\kappa\to\infty} r(\kappa)=0$.
Suppose that $\{\kappa_j\}$ is a sequence such that $\kappa_j\to\infty$ and that
$\{\nb_{j}\}$
is a  sequence  in $H^1(\Omega;\mathbb S^2)$ satisfying the following estimates,
\begin{align}\label{eq:apriori_n*}
\|\Div\nb_{j}\|_2+\|\curl\nb_{j}+\tau\nb_{j}\|_2\leq
r(\kappa_j)\,,\quad \text{ and } \quad
\|D \nb_{j} \|_2\leq C.
\end{align}
Then there exist a subsequence $\{(\psi_{j_s},\nb_{j_s})\}$ and a vector field $\nb_0$ in the set $C(\tau)$ introduced in \eqref{eq:C(tau)} such that,  as $s\to\infty$,
$$
\begin{aligned}
&\nb_{j_s}\to\nb_0\quad{\rm in~}H^1_{\rm loc}(\Omega,\R^3)\,,\\
&\nb_{j_s}\to\nb_0\quad{\rm in ~}L^p(\Omega;\R^3)\,,\quad 1\leq p<\infty\,,\\
&\nb_{j_s}\to\nb_0\quad{\rm in ~}W^{1,r}(\Omega;\R^3)\,,\quad 1\leq r<2\,.
\end{aligned}
$$
\end{lemma}
\begin{proof}
The  proof is  similar to that  of Theorem~1.1 in \cite{HP08}.
We have from \eqref{eq:apriori_n*} and the fact that $\nb$ takes values in $\mathbb S^2$ that $\{\nb_j\}$ is bounded in $H^1(\Omega,\R^3)$, hence there exists a subsequence $\{j_s\}$ such that $\nb_{j_s}$  converges weakly to a vector field $\nb_0$ in $H^1(\Omega,\R^3)$.
By the compact embedding of $H^1(\Omega,\R^3)$ into $L^p(\Omega,\R^3)$, $2\leq p\leq 6$, the convergence is strong in these $L^p(\Omega,\R^3)$. Passing to a further subsequence, the convergence holds a.e. in $\Omega$, hence $\nb_0$ inherits from $\nb_{j_s}$ the  constraint $|\nb_0|=1$ a.e. in $\Omega$.
By the uniform boundedness in $L^{\infty}$ and the convergence in $L^2(\Omega)$ (and the finiteness of the volume of $\Omega$ when $1\leq p<2$), we get the convergence in $L^p$ for all $1\leq p<\infty$ from the H\"{o}lder inequality.

By \eqref{eq:apriori_n*}, $\|\ddiv\nb_{j_s}\|_2\to 0$ and $\|\curl\nb_{j_s}+\tau\nb_{j_s}\|_2\to 0$. Hence
$$\Div\nb_0=0 \quad{\rm and}\quad \curl\nb_0+\tau\nb_0=\0\quad{\rm in~}\Omega\,.
$$
We conclude that $\nb_0\in\mathcal C(\tau)$.

Next we prove that $\nb_{j_s}\to\nb_0$ in $H^1_{\rm loc}(\Omega,\R^3)$  as $s\to\infty$.
Let $D$ be an open set such that $\overline{D}\subset \Omega$. Using local elliptic regularity of the $\curl$-$\Div$ system (the inequality (A.1) in \cite{HP08}), we may write,
\begin{multline}
\label{eq:(A.1)}
\|\nb_{j_s}-\nb_0\|_{H^1(D)}
\leq\\ C(D,\Omega)\Big\{\|\Div(\nb_{j_s}-\nb_0)\|_{L^2(\Omega)}
+\|\curl(\nb_{j_s}-\nb_0)\|_{L^2(\Omega)}
+\|\nb_{j_s}-\nb_0\|_{L^2(\Omega)}\Big\}\,,
\end{multline}
where $C(D,\Omega)$ solely depends on the domains $D$ and  $\Omega$. The convergence follows in light of the convergence in $L^2(\Omega,\R^3)$ and the inequality \eqref{eq:apriori_n*}.

To finish the proof of Lemma~\ref{lem:HP}, we need to prove that
$\nb_{j_s}\to\nb_0$  as $s\to\infty$ in $W^{1,r}(\Omega,\R^3)$ for any $1\leq r<2$.
By the convergence established  in $L^r(\Omega,\R^3)$, we need only prove the convergence of $D\nb_{j_s}$ in $L^r(\Omega,\R^3)$.
Let $\varepsilon_0>0$, $0<\varepsilon<\varepsilon_0$ and
$$\Omega_\varepsilon=\{x\in\Omega~:~{\rm dist}(x,\partial\Omega)>\varepsilon\}.
$$
We may select $\varepsilon_0$ sufficiently small such that, for all $\varepsilon\in(0,\varepsilon_0)$, $\Omega_\varepsilon$ is a non-empty open set. Smoothness of the boundary of $\Omega$ ensures that,
$$|\Omega\setminus\Omega_\varepsilon|=O(\varepsilon)\,,\quad (\varepsilon\to0_+)\,.
$$
Now, since $\nb_{j_s}$ is bounded in $H^1(\Omega,\R^3)$, there exists $C>0$ such that, for all $\varepsilon\in(0,\varepsilon_0)$,
$$\int_{\Omega\setminus\Omega_\varepsilon}|D\nb_{j_s}-D\nb_0|^r\,dx\leq C\varepsilon^{(2-r)/2}\,.
$$
By H\"{o}lder's inequality and the local $H^1$-convergence, for all $\varepsilon\in(0,\varepsilon_0)$,
\begin{align}
\int_{\Omega_\varepsilon}|D\nb_{j_s}-D\nb_0|^r\,dx
\leq C_{\varepsilon} \|\nb_{j_s}-\nb_0\|_{H^1(\Omega_\varepsilon)}^r\to 0\q\text{as }s\to\infty.
\end{align}
Thus, by taking the two successive limits, $s\to\infty$ and then $\varepsilon\to0_+$, we get that,
$$\limsup_{s\to\infty} \int_{\Omega}|D\nb_{j_s}-D\nb_0|^r\,dx\leq0\,.
$$
\end{proof}

\begin{lemma}\label{lem:lb-OFen}
Let $\tau>0$. If $\nb\in H^1(\Omega;\mathbb S^2)$,  then with $\mathcal F_N^+(\nb)$ from \eqref{eq:OFPE},
$$\mathcal F_N^+(\nb)\geq \min(K_1,K_2,K_3)\int_\Omega\Big\{|\Div\nb|^2+|\curl\nb+\tau\nb|^2\Big\}\,dx\,,
$$
and
$$\int_\Omega\Big\{{\rm tr}(D\nb)^2-|\Div\nb|^2\Big\}\,dx
\geq \int_\Omega\Big\{|D\nb|^2-|\Div\nb|^2\Big\}\,dx
-2\int_\Omega|\curl\nb+\tau\nb|^2\,dx-2|\Omega|\tau^2\,.
$$
\end{lemma}

\begin{proof}
This follows from the following  three pointwise identities:
\begin{align*}
|\curl \nb + \tau \nb |^2&=
|(\curl \nb )\cdot \nb + \tau |^2 +  |(\curl \nb) \times \nb |^2\,,\\
{\rm tr}(D\nb)^2&=|D\nb|^2-|\curl\nb|^2\,,\\
|\curl\nb|^2&=|\curl\nb+\tau\nb|^2+\tau^2-2\tau\nb\cdot(\curl\nb+\tau\nb)
\leq2|\curl\nb+\tau\nb|^2+2\tau^2\,.
\end{align*}
\end{proof}

\begin{prop}\label{prop:HP}
Let $c_2>c_1>0$ be constants.
Suppose that the Assumption~\ref{assumption:A} is satisfied and
$$c_1\kappa^2\leq K_2+K_4\leq c_2\kappa^2\,.$$
Then there exist constants $\kappa_0>0$, $q>0$, $\tau>0$  and $C>0$ such that, if $\kappa\geq\kappa_0$ and  $(\psi,\nb)$
is a minimizer of the functional in \eqref{eq-LdeG}, then
\begin{align}
\label{est3.1}
\|\Div\nb\|_2+\|\curl\nb+\tau\nb\|_2\leq
\sqrt{\frac{C\kappa^2}{2\,e(\kappa)}}
\quad
\text{ and } \quad
\|D \nb \|_2\leq C\tau.
\end{align}
\end{prop}
\begin{proof}
Notice that   Lemma~\ref{lem:lb-OFen} provides us with a lower bound of the energies $\mathcal F_N^+(\nb)$ and $\mathcal L(\nb)$ introduced in \eqref{eq:OFPE} and \eqref{eq:OFNL} respectively.
With this lower bound, the assumptions on the $K_i$,
and the assumption on $e(\kappa)\gg\kappa^2$, we can
estimate the energy in \eqref{eq-OFE} for large values of $\kappa$
as follows,
\begin{equation}\label{eq-OFE:lb1}
\mathcal F_N(\nb)
\geq \frac{e(\kappa)}2\int_\Omega|{\rm div}\nb|^2+
|\curl \nb + \tau \nb |^2\,dx+c_1\kappa^2\int_\Omega|D \nb|^2\,dx-2c_2|\Omega|\tau^2\kappa^2\,.
\end{equation}
Next, by writing
$$-|\psi|^2+\frac12|\psi|^4=\frac12(1-|\psi|^2)^2-\frac12\geq-\frac12\,,$$
we observe that the energy in \eqref{eq:GL} can be estimated from below as follows,
\begin{equation}\label{eq:lbGL}
\mathcal G(\psi,\nb)\geq -\frac{|\Omega|}2\kappa^2\,.
\end{equation}
Now we insert the estimates in \eqref{eq-OFE:lb1} and \eqref{eq:lbGL} into \eqref{eq-LdeG} to obtain,
\begin{equation}\label{eq:lbLdeG}
\mathcal E(\psi,\nb)\geq \frac{e(\kappa)}2\int_\Omega|{\rm div}\nb|^2+
|\curl \nb + \tau \nb |^2\,dx+c_1\kappa^2\int_\Omega|D \nb|^2\,dx-c'_2\kappa^2\,,
\end{equation}
where $c'_2=|\Omega|/2+2c_2\tau^2|\Omega|$.

Let $\nb_\tau\in\mathcal C(\tau)$ and
let $(\psi,\nb)$ be a minimizer of the functional in \eqref{eq-LdeG}.
We have $\mathcal E(\psi,\nb)\leq \mathcal E(0,\nb_\tau)=0$. Thus, we infer the inequalities of \eqref{est3.1} from \eqref{eq:lbLdeG}.
\end{proof}

\begin{proof}[Proof of Theorem~\ref{thm:HP}]
By Assumption~\ref{assumption:A}, we know that
$\lim_{\kappa\to\infty} \kappa^2/e(\kappa)=0$.
Theorem~\ref{thm:HP} therefore follows from Proposition~\ref{prop:HP} and Lemma~\ref{lem:HP}.
\end{proof}

\section{Preliminaries}

\subsection{An $L^\infty$ bound for order parameters}

Let $(\psi,\nb)$ be a minimizer of the energy ${\mathcal E}(\psi,\nb)$ in \eqref{eq-LdeG}.
Writing the Euler-Lagrange equation for ${\mathcal E}(\psi,\nb)$, we find that the function $\psi$ is a weak
solution\footnote{Here we have omitted the equation and boundary condition for $\nb$.} of
\begin{equation}\label{eq:EL-psi}
\left\{\aligned
-&\nabla_{q\nb}^2\psi=\kappa^2(1-|\psi|^2)\psi\quad &{\rm
in~}\Omega\,,\\
&\mathbf N\cdot\nabla_{q\nb}\psi=0\quad &\text{on }\partial \Omega,
\endaligned\right.
\end{equation}
where $\mathbf N$ is the interior unit normal vector
on $\partial\Omega$.
Repeating the argument in \cite{DuGP}, we get the following uniform bound.
\begin{lemma}
Suppose that $\nb\in H^1(\Omega;\mathbb S^2)$ and $\psi$ is a weak solution of \eqref{eq:EL-psi}. Then,
\begin{equation}\label{eq-Linfty}
\|\psi\|_\infty\leq
1\,.
\end{equation}
\end{lemma}

\subsection{A spectral estimate}
We will need the following well-known estimate in \eqref{eq:spect-est}. It follows, for instance, from the min-max principle and the fact that, if $\mathbf B=\curl\Ab$ is a non-zero constant vector, then  the spectrum of the magnetic Laplacian
$$-(\nabla-i\Ab)^2\quad{\rm in~}L^2(\R^3)$$
is  the whole interval $[|\mathbf B|,\infty)$.

\begin{lemma}
Let $\Ab$ be a vector potential with $\mathbf B=\curl\Ab$  constant in $\R^3$.
Let $u\in L^2(\R^3)$ with $(\nabla-i\Ab)u\in L^2(\R^3)$.
Then
\begin{equation}\label{eq:spect-est}
\int_{\R^3}|(\nabla-i\Ab)u|^2\,dx\geq  |\mathbf B|\int_{\R^3}
|u|^2\,dx\,,
\end{equation}
\end{lemma}

\subsection{Reduced Ginzburg-Landau energy}\label{sec:redGL}

Here we recall the definition of the boundary energy $E(\bb,\nu)$. This energy was
constructed in \cite{FKP3D} as the limit of a specific reduced
Ginzburg-Landau energy.

Let $\nu\in[0,\pi/2]$ and $\Ab_\nu$  be the magnetic potential
\begin{equation}\label{eq-3D-Eb}
\Ab_\nu(x)=(0,0,x_1 \cos\nu+x_2 \sin\nu)\,,\quad x=(x_1,x_2,x_3)\in\R^3\,.
\end{equation}
Clearly, the (un-oriented) angle between the vector $\curl\Ab_\nu = (\sin \nu, -\cos \nu, 0)$ and the plane
$\{x_1=0\}$ is $\nu$ (hence the geometry of the  flat boundary of the half space is
involved).

For each $\ell>0$, we introduce the domains,
\begin{equation}\label{eq-K-ell}
K_\ell=(-\ell,\ell)\times(-\ell,\ell)\,,\quad  U_\ell=(0,\infty)\times K_\ell\,,
\end{equation}
and the space,
\begin{equation}\label{eq-domain-ell}
\mathcal S_\ell=\{u\in L^2(U_\ell,\C)~:~(\nabla-i\Ab_\nu)u\in L^2(U_\ell,\C^3)
\,,~u=0\text{ on }(0,\infty)\times \partial K_\ell\}\,.
\end{equation}
Let $\bb\in(0,1]$ be a given constant---in the later use in the paper we have $\bb = 1/b$, with $b$ the constant from Assumption~\ref{assumption:A} i.e., $\bb$ is the ratio of $\kappa^2$ to the magnitude of the magnetic field. If $u\in\mathcal S_\ell$, we
define the reduced Ginzburg-Landau functional,
\begin{equation}\label{eq-Rgl}
\mathcal G_{\bb,\nu;\ell}(u)=
\int_{U_\ell}\left\{|(\nabla-i\Ab_\nu)u|^2
-\bb|u|^2+\frac{\bb}2|u|^4\right\}\,dx\,.
\end{equation}
Associated with $\mathcal G_{\bb,\nu;\ell}$ is the ground state
energy,
\begin{equation}\label{eq-Rgs}
d(\bb,\nu;\ell)=\inf_{u\in \mathcal S_\ell} \mathcal G_{\bb,\nu;\ell}(u)\,.
\end{equation}
From \cite[Theorem 3.6]{FKP3D} we know that a minimizer $u$ of $\mathcal G_{\bb,\nu;\ell}$ exists and that it satisfies the decay estimate\footnote{For $\bb<1$ exponential decay estimates would be possible, but for $\bb=1$ only weak decay is expected.}, which we recall for later use,
\begin{align}\label{eq:p-decay}
\int_{U_{\ell}} x_1^p |u(x)|^2\,dx \leq C_p \ell^2,
\end{align}
for all $0<p<1$.
In \cite{FKP3D}, it is proved that the limit of
$d(\bb,\nu;\ell)/\ell^2$ exists (and is finite) as $\ell\to\infty$.
Thus we define,
\begin{equation}\label{eq:E-FKP3D}
E(\bb,\nu\big)=\lim_{\ell\to\infty}\frac{d(\bb,\nu;\ell)}{(2\ell)^2}\,.
\end{equation}
An explicit formula for $E(\bb,\nu)$ is not available, but we know the following facts outlined in Remark~\ref{rem:useful} below (see \cite{FKP3D}):
\begin{remark}
\label{rem:useful}~
\begin{itemize}
\item $E(\bb,\nu)$ is a continuous function of $(\bb,\nu) \in (0,1] \times [0, \pi/2]$\,.
\item Let $\zeta(\nu)$ be the lowest eigenvalue of the Schr\"odinger
operator in the half-plane
$$\mathcal P_\nu=-(\nabla-i\mathbf F_\nu)^2\quad {\rm
in~} L^2(\R_+\times\R,\C)\,,
$$
where
$$\mathbf
F_\nu(x)=(0,x_1 \cos\nu+x_2 \sin\nu).
$$
Let
$\zeta(0):=\Theta_0$. It is known that $\Theta_0\sim0.59$ and that
$\zeta:[0,\pi/2] \to[\Theta_0,1]$ is a continuous and increasing
bijection. Furthermore, $E(\bb,\nu)$ vanishes when
$\zeta(\nu)\geq\bb$ and $E(\bb,\nu)<0$ otherwise. In particular,
$E(\Theta_0,\nu)=0$ for all $\nu\in[0,\pi/2]$.

\item There exist  two constants $\ell_0>0$ and $C>0$ such that, for all $\nu\in[0,\pi/2]$, $\bb\in(0,1]$ and $\ell\geq\ell_0$,
\begin{equation}\label{eq-E(nu,b)}
\frac{d(\bb,\nu,\ell)}{\ell^2}\leq E(\bb,\nu)\leq \frac{d(\bb,\nu,\ell)}{\ell^2}+\frac{C}{\ell^{2/3}}\,.
\end{equation}
\end{itemize}
\end{remark}

\subsection{Boundary coordinates}\label{sec:bndcod}
In order to analyze the influence of the boundary of the domain $\Omega$ on various quantities, we will carry out the computations in adapted coordinates near the
boundary---we will call these coordinates {\it boundary coordinates}. The specific choice of coordinates has been used in several contexts. In \cite{hemo} and then in
\cite{R}, the boundary coordinates are used  to estimate the ground state energy of a magnetic
Schr\"odinger operator with large magnetic field (or with small
semi-classical parameter). In \cite{FKP3D}, the boundary coordinates were useful in the computation of the ground state energy of the three dimensional Ginzburg-Landau functional. Here, we will use these coordinates in the same manner as in \cite{FKP3D}.

Since the boundary $\partial\Omega$ is $C^{2,\alpha}$-smooth, for every
point $p\in\partial\Omega$, there exist a neighborhood $\widetilde
V_p\subset \R^3$  of $p$, an open set $\widetilde U_p$ in $\R^2$ and
a $C^{2,\alpha}$-diffeomorphism
\begin{equation}\label{eq:chart-p-Omega}
\widetilde \phi_p:\widetilde
V_p\cap\partial\Omega\to \widetilde U_p.
\end{equation}
Furthermore, there exists a $C^{1,\alpha}$ smooth, unit inward pointing normal
vector field $\Nb:
\partial \Omega \rightarrow {\mathbb R}^3$.

This allows us to define,
\begin{align*}
\widetilde{\Phi}_{p}^{-1}: (-\epsilon_p,\epsilon_p) \times \widetilde U_p &\rightarrow {\mathbb R}^3,\\
(y_1,y_2,y_3) &\mapsto \widetilde \phi_p^{-1}(y_2,y_3) + y_1 \Nb(\widetilde \phi_p^{-1}(y_2,y_3)).
\end{align*}
For $\epsilon_p$ sufficiently small, $\widetilde{\Phi}_{p}^{-1}$ is a diffeomorphism on its image. Furthermore, we may assume that (after possibly scaling $\widetilde \phi_p$ and shrinking $\widetilde
V_p$)
\begin{align}\label{eq:BoundDeriv}
\frac{1}{2} \leq | D \widetilde{\Phi}_{p}^{-1}(y)| \leq \frac{3}{2}, \qquad \forall y \in  (-\epsilon_p,\epsilon_p) \times \widetilde U_p\,,
\end{align}
and that $\widetilde{\Phi}_{p}^{-1}((-\epsilon_p,\epsilon_p) \times \widetilde U_p) = \widetilde
V_p$.

Clearly, the family $\{\widetilde V_p: p\in\p\O\}$ is an open cover of
$\partial\Omega$. By compactness of $\partial\Omega$, we may extract
a finite subcover $\{\widetilde V_{p_i}\}_{i=1}^{ n}$ such that
\begin{equation}\label{eq:cover-p-Omega}
\partial\Omega \subset \displaystyle\bigcup_{i=1}^n\widetilde V_{p_i}\,.
\end{equation}
In particular, we may assume that $\epsilon_{p_i} \equiv \epsilon > 0$ is independent of $i$.
The collection $(\widetilde V_{p_i},\widetilde{\phi}_{p_i})_{i=1}^n$ will be fixed once and for all. In particular, it does not depend on the various parameters ($\kappa, \tau,\ldots$) of our problem. Of course, there is considerable freedom in this choice, but that is not important for our purpose. Different choices can possibly lead to different constants in our error bounds but not affect the overall results.

For convenience in the later use of the coordinates, we proceed to define such boundary coordinates centered around an arbitrary boundary point $x_0$ which might not be one of the $p_i$'s.

Consider a point  $x_0\in\partial\Omega$. There exists $p_{i}$ such that $x_0\in \widetilde V_{p_{i}}$.
Let $V_{x_0}$ be a neighborhood of $x_0$ such that $V_{x_0}\subset
\widetilde V_{p_i}$. Let
$$U_{x_0}:=\widetilde\phi_{p_i}(V_{x_0}\cap
\partial\Omega)\q\text{and}\q
\phi_{x_0}:=\widetilde\phi_{p_i}\big|_{V_{x_0}\cap\partial\Omega}.
$$
Clearly, $\phi_{x_0}$ is a diffeomorphism defining  local boundary
coordinates $(y_2,y_3)$ in
$$W_{x_0}=V_{x_0}\cap \partial\Omega
$$
through the relation $\phi_{x_0}(x)=(y_2,y_3)$. After possibly performing a translation and a rotation (notice that this does not affect \eqref{eq:BoundDeriv}), we may suppose that $0\in U_{x_0}$,
$\phi_{x_0}(x_0)=0$ and that $\Nb(x_0) = (1,0,0)$. We may modify $\phi_{x_0}$ further so that,
\begin{equation}\label{eq:D-phi'}
D\phi_{x_0}(x_0)=I_2\,,
\end{equation}
where $I_2$ is the $2\times2$ identity matrix---namely \eqref{eq:D-phi'} holds
after we replace the map $\phi_{x_0}(x)$ by
$$\phi_{x_0}^{\rm new}(x)=\big(D\phi_{x_0}(x_0)\big)^{-1}\phi_{x_0}(x).
$$

Notice that, by the choice of $\widetilde{\phi}_{p_i}$,  we have
\begin{equation}\label{eq:D-phi*}
\frac{1}{3} \leq |D\phi_{x_0}(x)|+|(D\phi_{x_0}(x))^{-1}| \leq 3 \q \text{for all } x\in V_{x_0},\;\; x_0\in\p\O\,.
\end{equation}

We define the
coordinate transformation $\Phi_{x_0}$ as
\begin{equation}\label{eq:coo-tran}
 (x_1,x_2,x_3)=\Phi_{x_0}^{-1}(y_1,y_2,y_3)=\phi_{x_0}^{-1}(y_2,y_3)+y_1
{\Nb}\big(\phi_{x_0}^{-1}(y_2,y_3)\big).
\end{equation}
In light of \eqref{eq:D-phi'}, we have,
$$
D\Phi_{x_0}(x_0)=I\,,
$$
where $I$ is the $3\times3$ identity matrix.

Now the standard Euclidean metric $g_0=\sum_{j=1}^3 dx_j \otimes dx_j$ on $V_{x_0}$ is transformed to the new metric on $\Phi_{x_0}(V_{x_0})$:
\begin{align*}
g_0 &= \sum_{1\leq j,k\leq 3} g_{jk} dy_j \otimes dy_k \\
    &= dy_3\otimes dy_3+\sum_{2\leq j,k\leq 3} \Bigl[
       G_{jk}(y_2,y_3)
       -2y_1 K_{jk}(y_2,y_3)
       +y_1^2 L_{jk}(y_2,y_3)
       \Bigr]
       dy_j\otimes dy_k\,,
\end{align*}
where
\begin{align*}
G&=\sum_{2\leq k,j\leq 3}G_{jk}\, dy_j\otimes dy_k
 = \sum_{\substack{2\leq k,j\leq 3\\1\leq l\leq 3}}
         \Bigl\langle\frac{\partial x_l}{\partial y_j}
       ,\frac{\partial x_l}{\partial y_k}\Bigr\rangle\, dy_j\otimes dy_k\,,\\
K&=\sum_{2\leq k,j\leq 3}K_{jk}\, dy_j\otimes dy_k
 = \sum_{2\leq k,j\leq 3}-\Bigl\langle \frac{\partial {\Nb}_{x_0}}{\partial y_j}
       ,\frac{\partial x}{\partial y_k}\Bigr\rangle\, dy_j\otimes dy_k\,\\
L&=\sum_{2\leq k,j\leq 3}L_{jk}\, dy_j\otimes dy_k
 = \sum_{2\leq k,j\leq 3}\Bigl\langle \frac{\partial {\Nb}_{x_0}}{\partial y_j}
       ,\frac{\partial {\Nb}_{x_0}}{\partial y_k}\Bigr\rangle\, dy_j\otimes dy_k\,,
\end{align*}
are the first, second and third fundamental forms on
$\partial\Omega$. We denote by $g^{jk}$  the entries of the inverse matrix of $(g_{jk})$. Its
Taylor expansion at $x_0$, valid in the neighborhood $\Phi_{x_0}(V_{x_0})$, is
given by
\begin{equation}\label{eq:metricapprox}
\big(g^{jk}\big)_{1\leq j,k\leq 3} =
I
+
\begin{pmatrix}
0 & 0 & 0\\
0 & O(|y|) & O(|y|)\\
0 & O(|y|) & O(|y|)
\end{pmatrix}.
\end{equation}
Note that both $g_{jk}$ and $g^{jk}$ depend on $x_0$ and on the modification of $\phi_{x_0}$ used to make \eqref{eq:D-phi'} valid.

Similarly, the Lebesgue measure $dx$ on $V_{x_0}$ transforms into the measure on $\Phi_{x_0}(V_{x_0})$ by the formula
$$dx=[\det(g_{jk})]^{1/2}dy.
$$
For future use, we introduce the notation
\begin{align}\label{eq:JacIsJac}
\Jac(y) = \det(g_{jk})^{1/2}(y),\q  y\in\Phi_{x_0}(V_{x_0}),\q x_0\in\p\O.
\end{align}
This determinant has the Taylor expansion
\begin{equation}\label{eq:jacobian}
[\det(g_{jk})]^{1/2} = \Jac(y)  = 1 + O(|y|),
\end{equation}
which is valid in the neighborhood $\Phi_{x_0}(V_{x_0})$.

We may express the integrals over $\Omega$ and
$\partial\Omega$ using the $y$-coordinates  as follows. For all
$u\in L^2(\Omega)$, $v\in L^2(\partial\Omega)$, and $x_0\in\p\O$, using the above notation, we have
\begin{equation}\label{eq:tran}
\begin{aligned}
&\int_{V_{x_0}}|u(x)|^2\,dx=\int_{\Phi_{x_0}(V_{x_0})}\det(g_{jk})^{1/2}|u(\Phi_{x_0}^{-1}(y))|^2\,dy\\
&\int_{V_{x_0}\cap\partial\Omega}|v(x)|^2\,ds(x)
=\int_{U_{x_0}}\det(g_{jk})^{1/2}\Big|_{y_1=0}|v(\Phi_{x_0}^{-1}(0,y_2,y_3))|^2\,dy_2dy_3\,.
\end{aligned}
\end{equation}

In light of \eqref{eq:jacobian}, we may simplify the formulas
displayed in \eqref{eq:tran} as follows.
There exist two
constants $C>0$ and
$\delta_0>0$, independent of the choice of the point
$x_0\in\partial\Omega$, such that, if $0<\delta<\delta_0$ and if
$$\Phi_{x_0}(V_{x_0})\subset \big\{y=(y_1,y_2,y_3)\in\R^3~:~|y|<\delta\big\},
$$
then, for all $u\in L^2(\Omega)$ and $v\in L^2(\partial\Omega)$,
\begin{equation}\label{eq:tran-jac}
\left| \int_{V_{x_0}}|u(x)|^2\,dx-\int_{\Phi_{x_0}(V_{x_0})}|u(\Phi_{x_0}^{-1}(y))|^2\,dy\right|\leq C\delta\int_{\Phi_{x_0}(V_{x_0})}|u(\Phi_{x_0}^{-1}(y))|^2\,dy\,,
\end{equation}
and
\begin{multline}\label{eq:tran-jac-sur}
\left|
\int_{V_{x_0}\cap\p\O}|u(x)|^2\,ds(x)
-\int_{U_{x_0}}|u(\Phi_{x_0}^{-1}(0,y_2,y_3))|^2\,dy_2dy_3\right|\\
\leq C\delta\int_{U_{x_0}}|u(\Phi_{x_0}^{-1}(0,y_2,y_3))|^2\,dy_2dy_3\,.
\end{multline}

A magnetic potential
$$\Fb=(F_1,F_2,F_3)
$$
defined in cartesian coordinates in $V_{x_0}$ is transformed to a magnetic
potential
$\widetilde{\Fb}$ in $y$-coordinates in $\Phi_{x_0}(V_{x_0})$. To save notation we shall write the vector field $\widetilde{\Fb}$ in terms of its coefficients associated with the natural (Cartesian) basis:
$$\widetilde{\Fb}(y)=(\widetilde{F}_1(y),\widetilde{F}_2(y),\widetilde{F}_3(y)),
$$
where
\begin{equation}\label{eq-gauge-F}
\widetilde{F}_j(y)
 := \sum_{k=1}^3 F_k(\Phi_{x_0}^{-1}(y))\frac{\partial x_k}{\partial y_j}.
\end{equation}

If $\curl \Fb$  is constant and has magnitude equal to $1$, and if
there exists a constant $C>0$ such that,
\begin{equation}\label{eq-gauge-F:cst}
\sum_{|\alpha|\leq 2}|D^\alpha \Fb|\leq C\,,\q x\in V_{x_0},
\end{equation}
then a particular choice of a gauge transformation is constructed
in~\cite{R} so that, in the neighborhood $\Phi_{x_0}(V_{x_0})$, the new vector
potential $\widetilde{\mathbf{F}}$ satisfies,
\eq\label{eq:tF}
\widetilde{F}_1 = 0,\qq
\widetilde{F}_2 = O\big(|y|^2\big),\qq
\widetilde{F}_3 = y_1\cos\nu+y_2\sin\nu+ O\big(|y|^2\big).
\eeq
Here, $\nu=\nu(x_0)$ is the angle between the magnetic field $\curl
\Fb$ and the tangent plane of $\partial\Omega$ at the point $x_0$,
i.e.
$$\sin \nu = |\Nb(x_0) \cdot (\curl \Fb)(x_0)|\,.
$$
Recall that this equality follows from \eqref{eq:D-phi'}, hence it is true after a rotation of the
$(y_2,y_3)$ coordinates, namely after modification of the mapping $\phi_{x_0}$.

Notice that the constants implicit in the ${O}$ notation in
\eqref{eq:metricapprox}, \eqref{eq:jacobian} and \eqref{eq:tF} are
determined by  the domain $\Omega$ as well as the constant $C$ in
\eqref{eq-gauge-F:cst}, i.e. these quantities are independent of the
boundary point $x_0$ by compactness and regularity of $\partial
\Omega$.

If $u$ is a function  with support in a coordinate neighborhood $V_{x_0}$,
we may express the functional
$$\mathcal
E_0(u,\Fb)=\displaystyle\int_\Omega\left\{|(\nabla-i\Fb)u|^2\,dx-\kappa^2|u|^2+\frac{\kappa^2}2|u|^4\right\}dx
$$
explicitly in the new coordinates as follows,
\eq\label{eq-bnd-en}
\aligned
\mathcal E_0(u,\Fb)=&
\int_{\Phi_{x_0}(V_{x_0})} \det(g_{jk})^{1/2}\biggl[
  \sum_{1\leq j,k\leq 3} g^{jk}
  \bigl(\partial_{y_j}-i\widetilde{F}_j\bigr)\widetilde u \times
  \overline{\bigl(\partial_{y_k}-i\widetilde{F}_k\bigr)\widetilde u}\\
&\qqq\qqq\qqq\qqq  -\kappa^2 |\widetilde u|^2
  +\frac{\kappa^2}{2}|\widetilde u|^4\biggr]\,dy,
\endaligned
\eeq
where  $\tilde u$ denotes the function defined in the new variable $y$ as follows:%
$$\widetilde u = \Big(\exp(-iq\tau \beta_{x_0}) u\Big) \circ  \Phi_{x_0}^{-1},
$$
(with $\beta_{x_0}$ the gauge transformation necessary to pass to the ${\bf
\widetilde F}$ given in~\eqref{eq:tF}).

Using the boundary coordinates, we can give a uniform bound of the integral of a function defined in a tubular neighborhood of the boundary.
The bound will involve the Sobolev norm and the thickness of the boundary layer.

For $t>0$ small we denote
\eq\label{Omegat}
\Omega_t=\{x\in\Omega~:~{\rm dist}(x,\partial\Omega)<t\}.
\eeq

\begin{lemma}\label{lem:bndTrTh}
There exist two constants $t_0\in(0,1)$ and $C>0$ such that, for all
$t\in(0,t_0)$ and $u\in W^{1,1}(\Omega, \C)$, it holds,
$$\int_{\O_t}|u(x)|\,dx\leq C\,t\|u\|_{W^{1,1}(\Omega)}\,.
$$
If $u=0$ on $\partial\Omega$, then we get the improved estimate,
\begin{equation}\label{eq:est-u=0}
\int_{\O_t}|u(x)|\,dx\leq
C\,t\int_{\O_{2t}}|Du(x)|\,dx\,.
\end{equation}
\end{lemma}
\begin{proof}
Let $(\Phi_{x_j}, V_{x_j})_{j=1}^N$ be a finite collection of local coordinate maps as above such that $\partial \Omega \subset \cup_{j=1}^N V_{x_j}$. Choose $t_0$ so small that
$$
\Omega_{2t_0} \subseteq
\bigcup_{j=1}^N V_{x_j}.
$$
For simplicity we write $\Phi_j$ instead of $\Phi_{x_j}$, etc. in the rest of the proof.

For all $j$, the Jacobian of $\Phi_j$ satisfies \eqref{eq:jacobian}. That way, there exists a constant $C>0$ such that, for all $t\in(0,t_0)$ and $u\in C^\infty(\overline{\Omega})$, (recall that $\phi_j$ was the boundary part of the coordinate transform $\Phi_j$)
$$\int_{\Omega_t\cap V_j}|u(x)|\,dx\leq C\int_0^{2t}\int_{\phi_j(V_j\cap \partial \Omega)}|\widehat u_j(y_1,y_2,y_3)|\,dy_2\,dy_3 \,dy_1\,,$$
where
$$\widehat u_j(y_1,y_2,y_3)=\chi(y_1) u\left(\Phi_j^{-1}(y_1,y_2,y_3)\right),
$$
and $\chi$ is a smooth cut-off function satisfying that ${\rm supp}(\chi)\subset [0,2t_0)$
and $\chi=1$ in $[0,t_0]$. Let us introduce
$$
u_j(x)=\widehat u_j\big(\Phi_j(x)\big).
$$

We write, for all  $y_1\in[0,2t_0)$,
$$
\begin{aligned}
\int_{\phi_j(V_j\cap \partial \Omega)}|\widehat u_j(y_1,y_2,y_3)|\,dy_2\,dy_3&=\int_{\phi_j(V_j\cap \partial \Omega)}\left|\int_{y_1}^{2t_0} \partial_z \widehat u_j(z,y_2,y_3)\,dz\right|\,dy_2\,dy_3\\
&\leq  C\|u_j\|_{W^{1,1}(V_j)}\leq C_j\|u\|_{W^{1,1}(\Omega)}\,.
\end{aligned}
$$
Summing up over $j\in\{1,2,\cdots,N\}$, we get,
\begin{align*}
\int_{\Omega_t}|u(x)|\,dx=\sum_{j=1}^N\int_{\Omega_t\cap V_j}|u(x)|\,dx
\leq \sum_{j=1}^NC_j\int_0^{2t}\|u\|_{W^{1,1}(\Omega)}\,dy_1=Mt\|u\|_{W^{1,1}(\Omega)}\,,
\end{align*}
where $M=2\sum_{j=1}^NC_j$,
is a constant independent of $u$ and $t$.
By density, we get the inequality for all $u\in W^{1,1}(\Omega,\C)$ as announced in Lemma~\ref{lem:bndTrTh}.

If $u=0$ on $\partial\Omega$, then we may write,
$$
\begin{aligned}
\int_{\phi_j(V_j\cap \partial \Omega)}|u_j(y_1,y_2,y_3)|\,dy_2\,dy_3&=\int_{\phi_j(V_j\cap \partial \Omega)}\left|\int_{0}^{y_1} \partial_zu_j(z,y_2,y_3)\,dz\right|\,dy_2\,dy_3\\
&\leq  C\int_{\O_{2t}}|Du|\,dx
\leq C_j\int_{\O_{2t}}|Du|\,dx\,.
\end{aligned}
$$
Summing up over $j\in\{1,2,\cdots,N\}$, we get \eqref{eq:est-u=0}.
\end{proof}

We conclude by outlining the construction of a useful partition of unity on $\partial\Omega$.

\begin{lemma}\label{lem:p-unity}
Let $\delta>0$ and define for $\eta>0$ the set
$$O_\eta=\{(y_1,y_2,y_3)\in\R^3~:~0<y_1<\delta\,,~-\eta<y_1,y_2<\eta  \}\,.$$

There exist  constants $C>0$ and $\epsilon_0\in(0,1)$  such that, for all $\delta,\alpha\in(0,\epsilon_0)$ the following holds.

There exist a finite sequence of points $\{x_{l}\}_{l =1}^N \subset \partial \Omega$ (with $N$ possibly depending on $\delta, \alpha$) and a smooth partition of unity $\{ \tchi_l\}_{l=1}^N$, with $ \tchi_l\geq 0$, and
\begin{align*}
&\partial\Omega\subset \bigcup_{l=1}^N \Phi_{l}^{-1}(O_\delta)\,,\\
&\sum_{l=1}^N  \tchi_l^2(x)= 1\qquad  \text{\rm in } \quad \{x\,:\,{\rm dist}(x,\partial\Omega)<\delta\},\\
&\tchi_l\equiv 1\q\text{\rm in} \quad Q_{\delta,l}:=\Phi_l^{-1}\bigl(O_{(1-\alpha)\delta}\bigr),\\
&\sum_l |\nabla \tchi_l(x)|^2 \leq C(\alpha\delta)^{-2}.
\end{align*}
Here $\Phi_l=\Phi_{x_{l}}$ is as in \eqref{eq:coo-tran}.
\end{lemma}

\begin{proof}
In $\R^2$, we introduce the following partition of unity
$$\sum_l g_{l,\delta,\alpha}^2=1\quad{\rm and}\quad \sum_l|\nabla g_{l,\delta,\alpha}|^2\leq \frac{C}{\alpha\delta}~{\rm in~}\R^2\,,$$
where
$${\rm supp}\,g_{l,\delta,\alpha}\subset Q_\delta(y_{j,\delta,\alpha}))\,,\quad g_{l,\delta,\alpha}=1\q ~{\rm in~}\q Q_{(1-\alpha)\delta}(y_{l,\delta,\alpha}))$$
$$y_{l,\delta,\alpha}=(1-\alpha)\delta l\,,\quad l=(l_1,l_2)\in\mathbb Z^2\,,$$
and, for $u=(u_1,u_2)\in\R^2$, $Q_\delta(y)=(y_1-\delta,y_1+\delta)\times(y_2-\delta,y_2+\delta)\subset\R^2$ is the square of center $y$ and length $\delta$.

Let us introduce the set of indices
$$\mathcal J=\{l\in\mathbb Z^2~:~y_{l,\delta,\alpha}\in\bigcup_{i=1}^n \widetilde U_{p_i}\}\,,$$
where $\big(\widetilde U_{p_i}=\phi_{p_i}(\widetilde V_{p_i})\big)_{i=1}^n$ is the class of subsets of $\R^2$ satisfying \eqref{eq:chart-p-Omega} and \eqref{eq:cover-p-Omega}.

Using the diffeomorphism in \eqref{eq:chart-p-Omega}, we get a family of points $(x_{l})_{l\in\mathcal J}$ in $\partial\Omega$ defined as follows
$$x_{l}=\phi_{p_i}^{-1}(y_{l,\delta,\alpha})~{\rm if~}y_{l,\delta,\alpha}\in\widetilde U_{p_i}\,.$$
Now we define the functions
$$
\begin{aligned}
&
\chi_l(x)=g_{l,\delta,\alpha}\big(\Phi_{x_{l}}(p(x))\big)~{\rm in~}\Phi_{x_{l}}^{-1}(O_\delta)\,,\quad \chi_l(x)=0~{\rm outside~}\Phi_{x_{l}}^{-1}(O_\delta)\,, \\
&f(x)=\sum_{j\in\mathcal J}\chi_l^2(x)~{\rm in~}\{{\rm dist}(x,\partial\Omega)<\delta\}\,,
\end{aligned}
$$
where, for $x\in\Omega$ satisfying ${\rm dist}(x,\partial\Omega)\leq \epsilon_0$,  $p(x)\in\partial\Omega$ is the unique point in $\partial\Omega$ such that ${\rm dist}(x,\partial\Omega)={\rm dist}(x,p(x))$. Note that  $f(x)\geq 1$ for all $x$.

Now we define the partition of unity $(\widetilde\chi_{l}(x))$ as follows,
$$
\widetilde\chi_l(x)=\frac{\chi_l(x)}{\sqrt{f(x)}}\,.$$
\end{proof}

\subsection{Gauge transformation}\label{sec:gauge}

The Ginzburg-Landau energy in \eqref{eq:GL} is gauge invariant, i.e.
$$\mathcal G(\psi,\nb)=\mathcal G(e^{-iq\beta}\psi,\nb+\nabla\beta)
$$
for any real-valued $H^1$-function $\beta$. In order to estimate $\mathcal G(\psi,\nb)$, we will replace $\nb$ with $\nb_{\rm new}=\nb+\nabla\beta$,
such that the new field $\nb_{\rm new}$ produces small errors in the various calculations.
This is easy to do when $\nb=\nb_0$ is a fixed smooth field, but is hard to do when $(\psi,\nb)$ is simply a minimizer of the Landau-de\,Gennes energy in \eqref{eq-LdeG} and therefore possibly varies with the various parameters (e.g. $\kappa$). Notice that, the new `director field' $\nb+\nabla\beta$ will generally not satisfy the pointwise normalization $|\nb+\nabla\beta| = 1$.

\subsubsection{Gauge for fields in $\mathcal C(\tau)$}

In this section, we fix a vector field
$\nb_0\in\mathcal C(\tau)$. By definition of $\mathcal C(\tau)$ (recall \eqref{C(tau)}), we
may write,
\begin{equation}\label{eq-nb0}
\nb_0(x)=N_{\tau}^Q(x) = Q N_{\tau}(Q^t x)\,,
\end{equation} where $N_\tau$ has been given in \eqref{Ntau}
and $Q$ is an orthogonal matrix such that ${\rm det}\,Q=1$.

For all $x_0\in\overline\Omega$, define the magnetic potential,
\begin{equation}\label{eq:n0cst}
(\nb_0)_{\rm cst}=
\int_0^1 s(x-x_0)\times \nb_0\big(x_0\big)\,ds\,.
\end{equation}
Note that  $(\nb_0)_{\rm cst}$ generates a constant magnetic field,
\begin{equation}\label{eq-curl=curl}
\curl\,(\nb_0)_{\rm cst}=-\nb_0(x_0)=\tau^{-1}(\curl\nb_0)(x_0)\,.
\end{equation}
In Lemma~\ref{lem:gt1}, we explain how to pass from the field $\nb_0$ to the field $(\nb_0)_{\rm cst}$.

\begin{lemma}\label{lem:gt1}
There exists a constant $C>0$ such that, if
\begin{itemize}
\item $x_0\in\overline{\Omega}$,
\item $U\subset\overline{\Omega}$ is a simply connected domain,
\item $x_0\in U$ and $\delta:={\rm diam}(U)$\,,
\end{itemize}
then there exists a smooth  function $\bar f_0: U\to\R$ such that,
\begin{equation}\label{eq-gauge1}
|\nb_0(x)-\tau(\nb_0)_{\rm cst}(x)-\nabla \bar f_0|\leq C\delta^2\q\text{in } U.
\end{equation}
Here $(\nb_0)_{\rm cst}$ is the vector field introduced in \eqref{eq:n0cst}.
\end{lemma}
\begin{proof}
Let
$$\mathbf a_0(x)=\tau\int_0^1 s(x-x_0)\times \nb_0\big(s(x-x_0)+x_0\big)\,ds\,.$$
It is clear that
$$\curl\mathbf a_0=-\tau\nb_0=\curl\nb_0\quad {\rm and}\quad \curl\,(\nb_0)_{\rm
cst}=-\nb_0(x_0)\,.
$$
In the simply connected domain $U$ containing $x_0$, there exists a smooth function $\bar f_0$ such that, we have,
$$\mathbf a_0=\nb_0-\nabla \bar f_0\quad{\rm in~}U\,.$$
Using \eqref{eq:n0cst} and the smoothness of the vector field $\nb_0$, we have,
$$|\mathbf a_0-\tau(\nb_0)_{\rm cst}|\leq \|D\nb_0\|_\infty|x-x_0|^2\leq C |x-x_0|^2\quad{\rm in
~}U\,.
$$
\end{proof}
As we shall see in Lemma~\ref{lem:gt2} below, if the point $x_0\in\partial\Omega$, then it is
possible to apply a further gauge transformation to transform
$\mathbf (\nb_0)_{\rm cst}$ to a magnetic potential of the form in
\eqref{eq-3D-Eb}.
In the statement of Lemma~\ref{lem:gt2}, we will use the following notation:
\begin{itemize}
\item $\nu_0:=\nu(\nb_0; x_0)\in[0,\pi/2]$ is the non-oriented angle between $\nb_0(x_0)$ and the tangent plane to $\partial\Omega$ at $x_0$\,;
\item $\Phi_0$ is the coordinate transformation that
straightens a neighborhood $V_0$ of the point $x_0$ such that
$\Phi_0(x_0)=0$ (see Sec.~\ref{sec:bndcod}).
\item For every $x\in V_0$, $(x_1,x_2,x_3)$ are the
standard cartesian coordinates of $x$ in $\R^3$ and
$$(y_1,y_2,y_3)=\Phi_0^{-1}(x_1,x_2,x_3)\,,\quad y_1\geq0\,.$$
\item
For every vector field $a$ defined in $V_0$ by the cartesian coordinates $(x_1,x_2,x_3)$, we denote by $\widetilde a$ the corresponding vector field defined via the boundary coordinates $(y_1,y_2,y_3)=\Phi_0(x_1,x_2,x_3)$, i.e. $\widetilde a(y_1,y_2,y_3)=a(x_1,x_2,x_3)$.
\item $\Ab_{\nu_0}$ is the vector potential (in boundary coordinates) introduced in
\eqref{eq-3D-Eb}.
\end{itemize}

Now we can state Lemma~\ref{lem:gt2}:

\begin{lemma}\label{lem:gt2}
There exist  two constants $C,\delta_0>0$ such that, if
\begin{itemize}
\item $x_0\in\partial\Omega$,
\item $U\subset\overline{\Omega}$ is a simply connected domain,
\item $x_0\in U$ and $\delta:={\rm diam}(U)\in(0,\delta_0]$\,,
\end{itemize}
then there exists a smooth  function $\beta_0$ such that,
\begin{equation}\label{eq:bndgauge}
\big|\widetilde{(\nb_0)_{\rm cst}}-\left(\Ab_{\nu_0}+\nabla
\beta_0\right)\big|\leq C\delta^2\,,\quad {\rm
in}~\Phi_0(V_0)\,.
\end{equation}
Here $(\nb_0)_{\rm cst}$ is the vector field introduced in \eqref{eq:n0cst}.
\end{lemma}
\begin{proof}
Choose $\delta_0$ small enough such that $B(x_0,\delta_0)\cap \overline{\Omega}\subset V_0$. Hence, for  $\delta\in(0,\delta_0]$, $U\subset V_0$.

Notice that, in light of \eqref{eq-curl=curl}, $|\curl(\nb_0)_{\rm
cst}|=1$ and, the (non-oriented) acute angle  between the fields
$\nb_0$, $(\nb_0)_{\rm cst}$ and the tangent plane to
$\partial\Omega$ at $ x_0$ are equal, i.e.
$$\nu\big((\nb_0)_{\rm
cst};x_0\big)=\nu(\nb_0; x_0)=\nu_0.
$$
The equation $y_1=0$ defines  part of the boundary of
$\partial\Omega$.
 Since $\curl\,(\nb_0)_{\rm cst}$ is a constant vector and makes an angle $\nu_0$ with $\partial\Omega$, it follows from \eqref{eq:tF}
 that we may find a smooth function $\beta_0(y)$ such that
\eqref{eq:bndgauge} holds.
\end{proof}

\subsubsection{Gauge for $\mathbb S^2$-valued fields}\label{sec:GT}
Let $\nb\in H^1(\Omega;\mathbb S^2)$. We will describe a procedure
allowing us to go from the  field $\nb$ (with  a {\it variable}
$\curl$) to a  field $\nb_{\rm cst}$ (with  a {\it constant}
$\curl$).
The errors produced by this procedure will be uniformly controlled by
 $\|\curl\nb+\tau\nb\|_2$. ($\tau>0$ is supposed a fixed constant, hence we do not seek estimates that are valid uniformly with respect to $\tau$).

\subsubsection*{\it Approximation by a field in $\mathcal C(\tau)$}

\begin{lemma}\label{lem:gt-S2}
Let $C_0>0$. There exists a constant $C>0$ such that, if
\begin{itemize}
\item $\nb\in H^1(\Omega;\mathbb S^2)$,  $\nb_0\in\mathcal C(\tau)$\,, $x_0 \in \overline{\Omega}$,  $\delta>0$,
\item $Q_{\delta} \subset \overline{\Omega}$ is starshaped with respect to $x_0$  and $Q_{\delta} \subset B(x_0, C_0 \delta)$\,,
\end{itemize}
then, there exists a function $f\in H^1(Q_\delta)$ such that
\begin{equation}\label{eq:a=a-cst-int}
\|\nb-\tau(\nb_0)_{\rm cst}-\nabla f_0\|_{L^2(Q_\delta)}\leq
3\delta\sqrt{|\ln\delta|}\Big(\|\curl\nb+\tau\nb\|_{L^2(Q_\delta)}+\tau\|\nb-\nb_0\|_{L^2(Q_\delta)}\Big)
+C\delta^3\,.
\end{equation}
Here $(\nb_0)_{\rm cst}$ is the field defined as in \eqref{eq:n0cst}
\end{lemma}
\begin{proof}
Define the vector fields
\begin{align}
&{\mathbf a}(x)=-\int_\eta^1 s(x-x_0)\times \big(\curl\nb\big)\,\Big(s(x-x_0)+x_0\Big)\,ds\,,\label{eq:mp-a1}\\
&\mathbf a^0(x)=-\int_\eta^1 s(x-x_0)\times \big(\curl\nb_0\big)\,\Big(s(x-x_0)+x_0\Big)\,ds\,,\label{eq:mp-a01}\\
&\mathbf c(x)=(\nb-\nb_0)\Big(\eta(x-x_0)+x_0\Big)\,.
\end{align}
It is easy to check that in $Q_{\delta}$,
\eq\label{eq:curl***}
\aligned
&\curl {\mathbf
c}(x)=\eta\curl(\nb-\nb_0)\Big(\eta(x-x_0)+x_0\Big),\\
& \curl( {\mathbf a}-{\mathbf a}^0)(x)=\curl(\nb-\nb_0)(x)-\eta\curl\mathbf c(x) \,.
\endaligned
\eeq
Consequently, since $Q_\delta$ is simply-connected, there exists a  function $f_0\in H^1(Q_\delta)$ such
that,
\begin{equation}\label{eq:curl;f}
\nb-\nb_0-\nabla f_0={\mathbf a}-{\mathbf a}^0+\eta\mathbf c\quad {\rm in}~Q_\delta\,.
\end{equation}
Now, we estimate $\|\mathbf a-\mathbf a^0\|_{L^2(Q_\delta)}$ and
$\|\mathbf c\|_{L^2(Q_\delta)}$. We have, for all $x\in Q_\delta$,
$$|\mathbf a(x)-\mathbf a^0(x)|^2\leq \delta^2\int_\eta^1
s^2\Big|\big(\curl\nb-\curl\nb_0\big)\,\Big(s(x-x_0)+x_0\Big)\Big|^2\,ds\,$$
and, performing the change of variable $y=s(x-x_0)+x_0$ (using that $Q_\delta$ is star-shaped),
$$\aligned
\int_{Q_\delta}|\mathbf a(x)-\mathbf a^0(x)|^2\,dx
\leq & \delta^2\int_\eta^1\frac1s\int_{
Q_{\delta}}\Big|\big(\curl\nb-\curl\nb_0\big)(y)\Big|^2\,dy\,ds\\
\leq&
\delta^2|\ln\eta|\,\|\curl(\nb-\nb_0)\|_{L^2(Q_\delta)}^2\,.
\endaligned
$$
Using the fact that $\curl\nb_0=-\tau\nb_0$,  the norm of
$\curl(\nb-\nb_0)$ can be estimated by the triangle inequality,
$$
\|\curl(\nb-\nb_0)\|_{L^2(Q_\delta)}
\leq \|\curl\nb+\tau\nb\|_{L^2(Q_\delta)}+\tau\|\nb-\nb_0\|_{L^2(Q_\delta)}\,.$$
This yields
\begin{equation}\label{eq:a=a-app}
\|\mathbf a-\mathbf a^0\|_{L^2(Q_\delta)}\leq \delta\sqrt{|\ln\eta|}\Big(\|\curl\nb+\tau\nb\|_{L^2(Q_\delta)}+\tau\|\nb-\nb_0\|_{L^2(Q_\delta)}\Big)\,.
\end{equation}
At the same time,
the identities $|\nb|=|\nb_0|=1$ yield,
$$\|\mathbf c\|^2_{L^2(Q_\delta)}\leq
4|Q_{\delta}|\leq C\delta^{3}\,,$$ where $C$ is a constant independent of $\delta$ and $\eta$. Inserting this estimate and the
one in \eqref{eq:a=a-app} into \eqref{eq:curl;f}, we deduce that,
\begin{equation}\label{eq:a=a-cst'}
\|\nb-\nb_0-\nabla f_0\|_{L^2(Q_\delta)}\leq
\delta\sqrt{|\ln\eta|}\Big(\|\curl\nb+\tau\nb\|_{L^2(Q_\delta)}+\tau\|\nb-\nb_0\|_{L^2(Q_\delta)}\Big)
+C\eta\delta^{3/2}\,.
\end{equation}
We can choose $\eta=\delta^{3/2}$ and get,
\begin{equation}\label{eq:a=a-cst}
\|\nb-\nb_0-\nabla f_0\|_{L^2(Q_\delta)}\leq
3\delta\sqrt{|\ln\delta|}\Big(\|\curl\nb+\tau\nb\|_{L^2(Q_\delta)}+\tau\|\nb-\nb_0\|_{L^2(Q_\delta)}\Big)
+C\delta^3\,.
\end{equation}
 In light of
\eqref{eq-gauge1}, we may modify the function $f_0$ in
\eqref{eq:a=a-cst} so that  \eqref{eq:a=a-cst-int} holds.
\end{proof}

\subsubsection*{\it Approximation by a constant field}
In the interior of $\Omega$, we can pass from a general $\nb\in H^1(\Omega;\mathbb S^2)$ to a vector field with constant $\curl$ in a manner  more efficient than   Lemma~\ref{lem:gt-S2}, in the sense better errors are produced. This is the purpose of Lemma~\ref{lem:gt-S2*} below.
\begin{lemma}\label{lem:gt-S2*}
There exists a constant $C>0$ such that, if
\begin{itemize}
\item $\nb\in H^1(\Omega;\mathbb S^2)$, $\nb_0\in\mathcal C(\tau)$, $x_0\in\Omega$, $\ell\in(0,1)$\,,
\item $Q_\delta\subset\Omega$ is a cube of side-length $\ell$ and center $x_0$\,,
\end{itemize}
then there exist  $f_0\in H^1(Q_\delta)$ and a vector field $\mathbf a_{\rm av}:\R^3\to\R^3$such that,
$$\curl\mathbf a_{\rm av}\text{ is constant,}$$
\begin{equation}\label{eq:curl:a-av**}
|\curl\mathbf a_{\rm av}|\geq 1-C\ell^6-C\ell\|D\nb\|_{L^2(Q_\ell)}\,,
\end{equation}
and
\begin{equation}\label{eq:gauge:blk**}
\|\nb-\tau\mathbf a_{\rm av}-\nabla f_0\|^2_{L^2(Q_\ell)}
\leq C\ell^2|\ln\ell|\Big(\|\curl\nb+\tau\nb\|^2_{L^2(Q_\ell)}+\ell^2\|D\nb\|_{L^2(Q_\ell)}\Big)+C\ell^9\,\,.
\end{equation}
\end{lemma}
\begin{proof}
We introduce the two  vector fields
$$\mathbf c=\nb\big(\ell^3(x-x_0)+x_0\big)\,,\quad\nb_{\rm av}=\frac1{|Q_\ell|}\int_{Q_\ell}\nb\,dx\,.$$
By  the Poincar\'e inequality, there exists a universal constant $C_0>0$ such that
\begin{equation}\label{eq:P-av}
\|\nb-\nb_{\rm av}\|_{L^2(Q_\ell)}\leq C_0\ell\|D\nb\|_{L^2(Q_\ell)}\,.
\end{equation}
Define $\mathbf a_{\rm av}$ as follows
$$\mathbf a_{\rm av}=-\int_\eta^1s(x-x_0)\times \nb_{\rm av}\,ds\,.$$
Clearly
\begin{equation}\label{eq:curl:a-av}
\curl\mathbf a_{\rm av}=(1-\ell^6)\nb_{\rm av}
\end{equation}
is constant and by \eqref{eq:P-av}
$$
|\nb_{\rm av}|=\|\nb_{\rm av}\|_{L^2(Q_\ell)}\geq \|\nb\|_{L^2(Q_\ell)}-\|\nb-\nb_{\rm av}\|_{L^2(Q_\ell)}\geq 1-C\ell\|D\nb\|_{L^2(Q_\ell)}\,,
$$
thereby giving \eqref{eq:curl:a-av**}.

Now, we introduce the `magnetic potential'
$$
\mathbf a=-\int_\eta^1s(x-x_0)\times (\curl\nb)\big(s(x-x_0)+x_0\big)\,ds\,.
$$
Clearly,
$$\curl\mathbf a=\curl\nb-\ell^3\curl\mathbf c\quad{\rm in~}Q_\ell\,.$$
This yields the existence of a function $f_0\in H^1(Q_\ell)$ such that
\begin{equation}\label{eq:gauge:blk*}
\nb-\mathbf a-\nabla f_0=\ell^3\mathbf c\quad{\rm in~}Q_\ell\,.
\end{equation}
On the other hand, we observe that,
\begin{align*}
\mathbf a(x)-\tau\mathbf a_{\rm av}(x)
&=-\int_{\ell^3}^1s(x-x_0)\times (\curl\nb+\tau\nb)\big(s(x-x_0)+x_0\big)\,ds\\
&\qquad+\tau\int_{\ell^3}^1s(x-x_0)\times (\nb-\nb_{\rm av}) \big(s(x-x_0)+x_0\big)\,ds\,.
\end{align*}
After applying the change of variables $y=s(x-x_0)+x_0$, we observe that,
$$\aligned
\|\mathbf a-\tau\mathbf a_{\rm av}\|^2_{L^2(Q_\ell)}
\leq& \ell^2|\ln(\ell^3)|\Big(\|\curl\nb+\tau\nb\|^2_{L^2(Q_\ell)}+\tau^2\|\nb-\nb_{\rm av}\|^2_{L^2(Q_\ell)}\Big)\\
\leq &C\ell^2|\ln(\ell^3)|\Big(\|\curl\nb+\tau\nb\|^2_{L^2(Q_\ell)}+\tau^2\ell^2\|D\nb\|^2_{L^2(Q_\ell)}\Big),\\
\endaligned
$$
Inserting this into \eqref{eq:gauge:blk*} then  using that $|{\mathbf c}|=1$, we obtain the conclusion in \eqref{eq:gauge:blk**}.
\end{proof}

\subsection{Energy of a boundary trial state}

In this subsection we shall estimate the local energy of a test function in a specific domain  $Q_\delta$.
For every $D\subset \overline\Omega$, we introduce the following
`local' energy in $D$,
\begin{equation}\label{eq:en-loc}
\mathcal G(\psi,\nb;D)=\int_D\left\{|(\nabla-iq\nb)\psi|^2-\kappa^2|\psi|^2+\frac{\kappa^2}2|\psi|^4\right\}\,dx\,.
\end{equation}

We suppose that $\tau>0$ and $\nb_0\in\mathcal C(\tau)$ are fixed (cf. \eqref{C(tau)}). Let $x_0\in\partial\Omega$ and $\Phi$ be the coordinate
transformation defined in \eqref{eq:coo-tran} that straightens a
neighborhood $V$ of the point $x_0$ such that $\Phi(x_0)=0$. Let
$\delta_0>0$ and $\delta\in(0,\delta_0)$. We select $\delta_0$
sufficiently small such that,
$$\forall~\delta\in(0,\delta_0)\,,\quad
(0,\delta)\times(-\delta,\delta)^2\subset \Phi(V)\,.
$$
From the discussion in subsection \ref{sec:bndcod}, the constant $\delta_0$ can be selected in a manner  independent of the variation of the point
$x_0\in\partial\Omega$.

Let
$Q_\delta=\Phi^{-1}\big((0,\delta)\times(-\delta,\delta)^2\big)$,
 $(\nb_0)_{\rm
cst}$ be the vector field in \eqref{eq:n0cst} and $\beta_0$ be the
function in \eqref{eq:bndgauge}. Notice that the vector field
$(\nb_0)_{\rm cst}$ generates a constant magnetic field,
$\curl(\nb_0)_{\rm cst}=-\nb_0(x_0)$. Let $\nu_0=\nu\big(x_0;\nb_0\big)$
be the angle in $[0,\pi/2]$ between  the vector $\nb_0( x_0)$ and
the tangent plane to $\partial\Omega$ at $ x_0$. This is the same angle in $[0,\pi/2]$ between  the vector $\curl(\nb_0)_{\rm cst}$ and
the tangent plane to $\partial\Omega$ at $ x_0$.

Let $\bar\beta_0=\beta_0\circ\Phi^{-1}$, where $\beta_0$ is the function introduced in Lemma~\ref{lem:gt2}. We have,
\begin{equation}\label{eq:bndgauge0*}
\big|\widetilde{(\nb_0)_{\rm
cst}}-\left(\Ab_{\nu_0}+\nabla \beta_0\right)\big|\leq
C\delta^2\quad{\rm in~}Q_\delta\,,
\end{equation}
where $\Ab_{\nu_0}$ is the vector field defined in \eqref{eq-3D-Eb}.

Let
$$\bb=\frac1b,\q
\ell=\delta\sqrt{q\tau}=\delta\sqrt{b}\,\kappa\,,
$$
and $u$ be a minimizer of the functional $\mathcal
G_{\bb,\nu_{0},\ell}$ introduced in \eqref{eq-Rgl}. For all
$x=\Phi^{-1}(y) \in Q_\delta$, we define the {\it trial function}
$\psi(x;x_0)$ as follows,
\begin{equation}\label{eq:psi-x-x0}
\psi(x;x_0)=\exp\big(iq g_0(x)\big) u\big(y\sqrt{q\tau}\,\big)\,,
\end{equation}
where
$$g_0=\bar f_0+\beta_0\circ \Phi^{-1}$$
and $\bar f_0$ is the function introduced in Lemma~\ref{lem:gt1} and satisfying  \eqref{eq-gauge1}.
Note that $\psi(x;x_0)$ is well-defined in $Q_\delta$. Recall that $u=0$ on $(0,\infty)\times\p K_\ell$, see \eqref{eq-domain-ell}. Hence $\psi(x;x_0)=0$ on $\O\cap \p Q_\delta$.

Near $x_0$, we expect that  $\psi(x;x_0)$ is an approximation of the actual minimizing order parameter, that is why we refer to it as  {\it trial function}.

Computing the energy of $\psi(x;x_0)$ in $Q_\delta$ is easy by
converting to boundary coordinates and dilating the variables.

\begin{lemma}\label{lem:psi:x0}
Let $b>1$ be a fixed constant and
$\nb_0\in\mathcal C(\tau)$. There exists a
constant $C>0$ such that,
for all
$\eta\in(0,1)$, $
\delta\in(0,\delta_0)$, $\kappa\geq 1$,  $q\tau=b\kappa^2$ and $ x_0\in\partial\Omega$, the function
$\psi(x;x_0)$ in \eqref{eq:psi-x-x0} satisfies,
$$\frac{\mathcal G\big(\psi(\cdot;x_0),\nb_0;Q_\delta\big)}{|\overline{Q_\delta}\cap\partial\Omega|}\leq
(1+C\eta+C\delta)\sqrt{q\tau}\,E\big(b^{-1},\nu_0\big)+r\,,
$$
where {the constant} $r$ satisfies,
\begin{equation}\label{eq:r}
|r|\leq C\eta^{-1}q^2\delta^4+
C\sqrt{q\tau}\left(\delta+\eta+\eta^{-1}q\delta^4+\frac{|\ln\kappa|}{\delta\sqrt{q\tau}}\right)\,.
\end{equation}
 Here, $E(\cdot,\cdot)$
is the energy introduced in \eqref{eq:E-FKP3D}, and
$\nu_0=\nu(x_0;\nb_0)$ is the angle in $[0,\pi/2]$ between the
vector $\nb_0(x_0)$ and the tangent plane to $\partial\Omega$ at
$x_0$.

Furthermore, for $\frac{1}{\sqrt{q\tau}} \leq \epsilon \leq \delta$, we have
\begin{align}\label{eq:L2-bound}
\int_{\{x\in Q_{\delta}, \dist(x,\partial \Omega \geq \epsilon\}} |\psi(\cdot,x_0)|^2\,dx \leq C_p \frac{\delta^2}{ \sqrt{q\tau}(\sqrt{q\tau}\epsilon)^p},
\end{align}
for all $0<p <1$.
\end{lemma}
\begin{proof}
The estimate \eqref{eq:L2-bound} is immediate, using the decay of $u$ in the normal coordinate (see \eqref{eq:p-decay}).

For simplicity, we will omit $x_0$ from the notation of the trial
function and write $\psi=\psi(\cdot;x_0)$. Let $\eta\in(0,1)$. We
write by the Cauchy-Schwarz inequality,
\begin{align*}
\int_{Q_\delta}|(\nabla-iq\nb_0)\psi|^2\,dx
&\leq
(1+\eta)\int_{Q_\delta}\big|\big(\nabla-i\tau q\tau(\nb_0)_{\rm cst}\big)e^{-iq\bar f_0}\psi|^2\,dx\\
&\quad+C\eta^{-1}q^2\int_{Q_\delta}|\nb-\tau(\nb_0)_{\rm cst}-\nabla\bar f_0|^2|e^{-iq\bar f_0}\psi|^2\,dx\,.
\end{align*}
Now, using the bound $|\psi|\leq 1$ and the estimate in
\eqref{eq-gauge1}, we get further,
\begin{align*}
\int_{Q_\delta}|(\nabla-iq\nb_0)\psi|^2\,dx\leq (1+\eta)\int_{Q_\delta}\big|\big(\nabla-i\tau q\tau(\nb_0)_{\rm
cst}\big)e^{-iq\bar f_0}\psi|^2\,dx+C\eta^{-1}q^2\delta^6\,.
\end{align*}
We convert to integration in boundary coordinates (by using
\eqref{eq:jacobian} and \eqref{eq-bnd-en}). Using  the estimate in
\eqref{eq:bndgauge} and the Cauchy-Schwarz inequality, we may
write,
\begin{align*}
&\mathcal G(\psi,\nb_0;Q_\delta) \\
&~\leq
(1+C\eta+C\delta)\int_{\Phi(Q_\delta)}
\left\{|(\nabla-iq\tau\Ab_{\nu_0})u(y\sqrt{q\tau})|^2-\kappa^2|u(y\sqrt{q\tau})|^2
+\frac{\kappa^2}2|u(y\sqrt{q\tau})|^4\right\}\,dy\\
&\qquad+C\eta^{-1}q^2\delta^6+r_1\,,
\end{align*}
where
\begin{equation}\label{eq:r2}
r_1=C(\delta+\eta)\kappa^2\int_{\Phi(Q_\delta)}\Big\{|u(y\sqrt{q\tau})|^2+|u(y\sqrt{q\tau})|^4\Big\}\,dx+C\eta^{-1}q^2\delta^4\int_{\Phi(Q_\delta)}|u(y\sqrt{q\tau})|^2\,dy\,.
\end{equation}
Recall that $q\tau=b\kappa^2$ and $\ell=\delta\sqrt{q\tau}$. We
perform the change of variable $z=y\sqrt{q\tau}$ and use that the
function $u$ decays at infinity (see \cite[Thm.~3.6]{FKP3D}) to
write,
$$
\mathcal G(\psi,\nb_0;Q_\delta)
\leq (1+C\eta+C\delta)\frac{d(b^{-1},\nu_0;\ell)}{\sqrt{q\tau}}+\frac{C|\ln\kappa|}{\ell}\delta^2\sqrt{q\tau}+C\eta^{-1}q^2\delta^6+r_1\,,$$ and
\begin{equation}\label{eq:r2'}
r_1\leq
C(\delta+\eta)\delta^2\sqrt{q\tau}+C\eta^{-1}q\delta^6\sqrt{q\tau}\,.
\end{equation}
Thanks to \eqref{eq-E(nu,b)}  and the fact that
$\ell=\delta\sqrt{q\tau}$, we get,
$$
\mathcal G(\psi,\nb_0;Q_\delta) \leq
(1+C\eta+C\delta)\sqrt{q\tau}\,(2\delta)^2E(b^{-1};\nu_0)
+\frac{C|\ln\kappa|}{\ell}\delta^2\sqrt{q\tau}+C\eta^{-1}q^2\delta^6+r_1\,.$$
This finishes the proof of Lemma~\ref{lem:psi:x0}, thanks to
 \eqref{eq:r2'}, the boundedness of the function
$E(\cdot,\cdot)$ and the following estimate that results from
\eqref{eq:tran-jac-sur},
$$\Big|\,(2\delta)^2-|Q_\delta\cap\partial\Omega|\,\Big|\leq
C\delta^3\,.$$
\end{proof}

\section{Upper bound for the energy}\label{sec:ub}

In this section we derive an upper bound estimate of the value of $\mathcal E(\psi,\nb)$, where  $(\psi,\nb)$ is a minimizer of the  Landau-de\,Gennes energy $\mathcal E$ given in \eqref{eq-LdeG} and
the Assumption~\ref{assumption:A} is satisfied.

For every $x\in\partial\Omega$, recall the definition of  $E\big(\bb,\nu(x;\nb)\big)$ and $\Eground(\frac1b,\tau)$ (see \eqref{eq:E-FKP3D}, \eqref{E(bnu)} and \eqref{eq-F0(nb)}).

\begin{prop}\label{prop:ub}
Let $b>1$ and $\tau>0$ be  fixed constants. There exists a function ${\rm err}:[1,\infty)\to\R_+$ such that $\displaystyle\lim_{\kappa\to\infty}{\rm err}(\kappa)=0$ and the following is true. For all
$$\kappa\geq 1\,,\quad q\tau=b\kappa^2\,,\quad K_1,K_2,K_3\geq0\,,$$
the ground state energy in \eqref{eq-gs-LdeG} satisfies,
\begin{equation}\label{eq:ub}
\gse\leq  \sqrt{q\tau}\Eground\Big(\frac1b,\tau\Big)+ \kappa\,{\rm err}(\kappa)\,.
\end{equation}
\end{prop}
To prove Proposition~\ref{prop:ub}, we need:
\begin{lemma}\label{lem:ub**}
Let $b\in(1,\Theta_0^{-1})$, $\tau>0$ and $\nb_0\in\mathcal C(\tau)$ (defined in  \eqref{C(tau)}). There exist constants $\kappa_0\geq 1$, $C>0$ and $\eta_0\in(0,1)$ such that,  for all $\eta\in(0,\eta_0)$, $\kappa\geq \kappa_0$, there exists a function $\psi_{\rm trial}\in H^1(\Omega;\mathbb C)$ satisfying,
$$\limsup_{\substack{\kappa\to\infty\\q\tau=b\kappa^2}}\frac{\mathcal G(\psi_{\rm trial},\nb_0)}{\sqrt{q\tau}}\leq (1+C\eta)\tilde E\left(\frac1b,\nb_0\right)+C\eta\,.$$
Here,   $\mathcal G(\cdot,\cdot)$ and $E(\cdot,\cdot)$ are introduced in \eqref{eq:GL} and \eqref{E(bnu)}.
\end{lemma}

\begin{proof}[Proof of Proposition~\ref{prop:ub}]
Let $\nb_0\in\mathcal C(\tau)$ and $\psi_{\rm trial}$ be the trial function in Lemma~\ref{lem:ub**}. We may write the following
upper bound for the full Landau-de\,Gennes functional,
$$\mathcal E(\psi,\nb)=\gse\leq \mathcal E(\psi_{\rm trial},\nb_0)=\mathcal G(\psi_{\rm trial},\nb_0)\,,$$
using that the Oseen-Frank energy of $\nb_0$ vanishes, since $\nb_0\in\mathcal C(\tau)$.
Using the conclusion of Lemma~\ref{lem:ub**} we get
$$
\limsup_{\kappa\to\infty}\frac{\gse}{\sqrt{q\tau}}\leq
(1+C\eta)\tilde E\left(\frac1b,\nb_0\right)
+C\eta\,.
$$
The term in the left side above is independent of  $\eta$. Taking the limit as $\eta\to0_+$, we get
\begin{equation}\label{eq:ub**}
\limsup_{\kappa\to\infty}\frac{\gse}{\sqrt{q\tau}}\leq
\tilde E\left(\frac1b,\nb_0\right)
\,.
\end{equation}
 Since \eqref{eq:ub**} is true for all $\nb_0\in\mathcal C(\tau)$, then by definition of ${\Eground}(\frac1b,\tau)$ in \eqref{eq-F0(nb)}, we get
$$\limsup_{\kappa\to\infty}\frac{\gse}{\sqrt{q\tau}}\leq
\Eground\left(\frac1b,\tau\right)\,.$$
\end{proof}

\begin{proof}[Proof of Lemma~\ref{lem:ub**}]
The   proof consists of two parts. The first part is  devoted to the construction of the trial
function $\psi_{\rm trial}$ and the second part is devoted to  computation of the corresponding energy, $\mathcal G(\psi_{\rm trial},\nb_0)$.

\paragraph{Step 1. Splitting the boundary region into small disjoint boxes.}\

Let $\eta>0$ be small but fixed. We will choose another parameter
$\delta>0$ which will be specified as a negative power of $\kappa$
below.

Choose a finite collection of points $\{x_j: 1\leq j\leq m\} \subset \partial
\Omega$ such that
\begin{align*}
\forall j \neq k: \quad \eta/2 \leq \dist(x_j,x_k)\,,\quad
\forall j:\quad \min_{k\neq j} \dist(x_j,x_k) \leq 2\eta\quad \text{ and }~\partial\Omega\subset \bigcup_{j=1}^mB(x_j,4\eta)\,.
\end{align*}
Define $U_j$ as
$$
U_j=\{x\in\partial\Omega~:~\forall~k\not=j\,,~{\rm
dist}(x,x_j)<{\rm dist}(x,x_k)\}\,.
$$
Clearly the $U_j$'s are disjoint and
$$\partial \Omega = \bigcup_{j=1}^m\overline{U}_j.
$$

Next, we construct a family of sets that cover a tubular
neighborhood of $\partial\Omega$. That will be done by using  the
boundary coordinates $(y_1,y_2,y_3)$ introduced in
Sec.~\ref{sec:bndcod} (where the equation $y_1=0$ defines the corresponding part in
$\partial\Omega$). Let $\Phi_j$ be the coordinate transformation
that straightens a neighborhood $V_j$ of the point $x_j$ such that
$\Phi_j(x_j)=0$. We may assume that $U_j \subset V_j$ for all $j=1,\cdots, m$ (this amounts to selecting $\eta$ sufficiently small).
Let
$$O_j=\{x=\Phi_j^{-1}(y_1,y_2,y_3)~:~\Phi_j^{-1}(0,y_2,y_3)\in U_j\;\;~{\rm and}\;\;~0<y_1<\delta\}\,,\q
j=1,\cdots, m.
$$

We impose the following condition on $\delta$:
$$
\frac{1}{\sqrt{\kappa H}} \ll \delta \ll \eta\,.
$$
The  number $m$ of the sets $U_j$ is independent of
$\delta$ (but depends on $\eta$). Define
$$\widetilde O_j^{\rm 2D}=\Phi_j(O_j)\cap \{ y \in {\mathbb
R}^3 \,:\, y_1 = 0\}.
$$
For fixed $\delta$ and for each $1\leq j\leq m$ we may cover $\widetilde O_j^{\rm 2D}$ by a tiling of non-overlapping squares $\{K_{j,i}\}$, $i=1,\cdots, n_j$,
where $K_{j,i}$ is centered at
the point $y_{j,i}$ and has side-length $2\delta$ i.e., $K_{j,i} = y_{j,i}+[-\delta,\delta]^2$, with $y_{j,i} \in \{0\}\times 2\delta {\mathbb Z}^2$.

Let
$$\mathcal J_j=\{i~:~K_{j,i}\subset \widetilde O_j^{\rm 2D}\},\quad
N_j={\rm Card}\,\mathcal J_j\,,\q N=\sum_{j=1}^mN_j.
$$
Note that both $N_j$ and $N$ depend on $\delta$. We combine the
coordinate transformation $\Phi_j$ by a translation taking the point $y_{j,i}$ to $0$. Thus, we let
$\Phi_{j,i}$ be the resulting coordinate transformation defined by
$$\Phi_{j,i}^{-1}({\bf y}) = \Phi_j^{-1}(y_{j,i} + {\bf y}),
$$
for ${\bf y} = (y_1, y_2, y_3) \in (0,\delta) \times (-\delta, \delta)^2$.

As a result of the construction of the transformations $\Phi_{j,i}$,
 we get a tiling of (most of) the three dimensional domain
$O_j$ by the `cube-like' sub-domains
$$
Q_{j,i}=\Phi_{j,i}^{-1}\left((0,\delta)\times(-\delta,\delta)^2\right).$$
We restrict to the indices $(j,i)$ such that $1\leq j\leq m$ and $i\in\mathcal J_j$.
By construction the $Q_{j,i}$'s are non-overlapping.
Let
$x_{j,i}=\Phi_{j,i}^{-1}(0)$.

\paragraph{Step 2. Splitting the energy.}\

Our trial state will have the structure
\begin{align*}
\psi_{\rm trial}(x) = h_\delta(x)u_{\rm trial}(x)\,,\quad u_{\rm trial}(x)=\sum_{j=1}^n \sum_{i=1}^{N_j(\delta)}\psi^{\rm trial}_{j,i}(x),
\end{align*}
where
$$h_\delta(x)=h\left(\frac{{\rm dist}(x,\partial\Omega)}{\delta}\right)\,.
$$
Here $h$ is
a smooth cut-off function satisfying
$${\rm supp}\,h\subset [-1,1]\,,\quad 0\leq h\leq 1{\rm ~in~}\R\,,\quad h(x)=1 {\rm ~in}~[-1/2,1/2]\,,$$
and, for all $(j,i)$,
\begin{equation}\label{eq:trial;j,i}
\psi^{\rm trial}_{j,i}(x)=\psi(x;x_{j,i})\,,
\end{equation}
is the function introduced in \eqref{eq:psi-x-x0} with
$x_0=x_{j,i}$, $Q_\delta=Q_{j,i}$ and $\Phi=\Phi_{j,i}$.
Recall that $\psi(x;x_{j,i})$ is well-defined in $Q_{j,i}$ and $\psi(x;x_{j,i})=0$ on $\O\cap\p Q_{j,i}$. Hence we may extend it over $\O$ by letting it equal zero outside of $Q_{j,i}$.
The energy
of this function is estimated in Lemma~\ref{lem:psi:x0}.

Let us start by observing that
\eq
|u_{\rm trial}|\leq 1\q\text{in }\O\,.
\eeq
The function $u_{\rm
trial}$ inherits this bound from the definition of the functions
$\psi_{j,i}^{\rm trial}$ and the fact that these functions have
mutually disjoint supports.

From the disjoint supports of the summands and by the Cauchy-Schwarz inequality, we get for any $\eta \in (0,1)$,
\begin{align}
&{\mathcal G}(\psi_{\rm trial},\nb_0) = \sum_{j=1}^n \sum_{i=1}^{N_j(\delta)} \mathcal G(h_{\delta}\psi^{\rm trial}_{j,i},\nb_0;Q_{j,i}) \nonumber\\
&\leq (1+\eta)\sum_{j=1}^n \sum_{i=1}^{N_j(\delta)} \int_{Q_{j,i}}|h_\delta(\nabla-iq\nb_0)\psi^{\rm trial}_{j,i}|^2-\kappa^2|h_\delta \psi^{\rm trial}_{j,i}|^2+\frac{\kappa^2}2h_\delta^4|\psi^{\rm trial}_{j,i}|^4\,dx + \frac{C}{\eta \delta^{2}} \,dx\nonumber \\
&\leq (1+\eta)\sum_{j=1}^n \sum_{i=1}^{N_j(\delta)} \left\{\mathcal G(\psi^{\rm trial}_{j,i},\nb_0;Q_{j,i})
+ C(\kappa^{1/2} \delta^{3/2}+ \eta^{-1}\delta)\right\}.
\end{align}
Here we also used \eqref{eq:L2-bound} (with $p=1/2$ for concreteness),
that the volume of $Q_{j,i}$ is controlled by $C \delta^3$ and that $q\tau = b \kappa^2$.
We use the estimate in Lemma~\ref{lem:psi:x0} to estimate $\mathcal
G(\psi^{\rm trial}_{j,i},\nb_0;Q_{j,i})$ and find with the remainder $r$ from \eqref{eq:r},
\begin{multline}\label{eq:UpperSomething}
{\mathcal G}(\psi_{\rm trial},\nb_0)  \leq(1+C\eta+C\delta) \times\\
\sum_{j=1}^n
\sum_{i=1}^{N_j(\delta)}
\int_{Q_{j,i} \cap \partial \Omega} \sqrt{q\tau} E(b^{-1},\nu_{j,i}) + r + C(\kappa^{1/2} \delta^{-1/2} + \eta^{-1} \delta^{-1}) \,d\sigma(x),
\end{multline}
where $d\sigma$ is the surface measure, and where we used that $\sigma(Q_{j,i} \cap \partial \Omega)/\delta^2$ is bounded from above and below.

We choose
$$\delta=\kappa^{-4/5}.
$$
The remainder term $r$ now satisfies,
$$r\leq C( \eta^{-1} \kappa^{4/5} + |\ln \kappa| \kappa^{4/5} + \eta \kappa)\,,$$
and since the $Q_{j,i}$ are disjoint, we may estimate
$$
\sum_{j=1}^n
\sum_{i=1}^{N_j(\delta)}
\int_{Q_{j,i} \cap \partial \Omega} d\sigma(x) \leq \int_{\partial \Omega} d\sigma(x) = C < \infty.
$$
So \eqref{eq:UpperSomething} becomes
\begin{multline}\label{eq:Riemann}
{\mathcal G}(\psi_{\rm trial},\nb_0)  \leq
(1+C\eta+C\kappa^{-4/5})\,\sum_{j=1}^n
\sum_{i=1}^{N_j(\delta)}
\int_{Q_{j,i} \cap \partial \Omega} \sqrt{q\tau} E(b^{-1},\nu_{j,i}) \,d\sigma(x)\\
+ C (\eta^{-1} \kappa^{4/5} + \eta \kappa + \kappa^{9/10} ).
\end{multline}

Clearly, the sum in \eqref{eq:Riemann} is a `Riemann sum'
of the continuous function $\partial\Omega\ni x\mapsto
E\big(b^{-1},\nu(x,\nb_0\big)$. As a consequence of this, we may write,
$$\sum_{j=1}^n
\sum_{i=1}^{N_j(\delta)}|Q_{j,i}\cap\partial\Omega|E(\frac{1}{b},\nu_{j,i})\leq
\int_{\partial\Omega}E\Big(\frac{1}{b},\nu(x;\nb_0)\Big)\,ds(x)+o(1),
$$
as $\delta \rightarrow 0$.
We insert this into \eqref{eq:Riemann} to obtain,
\begin{multline}
\mathcal G(\psi_{\rm trial},\nb_0)
\leq
(1+C\eta+C\kappa^{-4/5})\sqrt{q\tau}\left(\int_{\partial\Omega}E\Big(\frac{1}{b},\nu(x;\nb_0)\Big)\,ds(x) + o(1)\right)\\
+C(\eta+\eta^{-1}\kappa^{-1/5}+\kappa^{-1/10})\sqrt{q\tau}\,.
\end{multline}
We take $\limsup_{\kappa\to\infty}$ to get,
$$\limsup_{\kappa\to\infty}\frac{\mathcal G(\psi_{\rm trial},\nb_0)}{\sqrt{q\tau}}\leq
(1+C\eta)\int_{\partial\Omega}E\Big(\frac{1}{b},\nu(x;\nb_0)\Big)\,ds(x)
+C\eta\,.
$$
Recalling the definition of $\tilde E(\cdot,\cdot)$ in \eqref{E(bnu)}, this finishes the  proof of Lemma~\ref{lem:ub**}.
\end{proof}

\section{Lower bound for the energy}\label{sec:lb}


 In this section, we derive a lower bound of the ground state energy in \eqref{eq-gs-LdeG} under the Assumptions~\ref{assumption:A} and \ref{assumption:A'}.

Let $D\subset\Omega$ be a regular set as in
Definition~\ref{def:dom}. We introduce the local Ginzburg-Landau
energy as follows (compare with \eqref{eq:GL}),
\begin{equation}\label{eq:GLloc}
\mathcal G(\psi,\nb;D)=\int_D\left\{|(\nabla-iq\nb)\psi|^2-\kappa^2|\psi|^2+\frac{\kappa^2}2|\psi|^4\right\}\,dx\,.
\end{equation}

The main result in this section is:

\begin{theorem}\label{thm:lb}
Suppose that the Assumptions~\ref{assumption:A} and \ref{assumption:A'} are satisfied, and let $(\psi_j,\nb_j)$ be a minimizer of the energy in \eqref{eq-gs-LdeG} corresponding to $\kappa=\kappa_j$ (and $q=q_j=b\kappa_j^2/\tau$ with $\tau>0$ and $b>1$ being fixed by Assumptions~\ref{assumption:A} and \ref{assumption:A'}).

There exist $\nb_0\in\mathcal C(\tau)$ and  a subsequence $\{\kappa_{j_s}\}$ such that,  for every {\bf regular} subset $D\subset \Omega$, and for every $h\in H^1(D)$ satisfying $\|h\|_\infty\leq 1$, it holds,
$$\mathcal G(h\psi_{j_s},\nb_{j_s};D)\geq\sqrt{q_{j_s}\tau}\int_{\overline{D}\,\cap\partial\Omega}
E\Big(\frac1b,\nu(x;\nb_0)\Big)\,ds(x)+ o(\sqrt{q_{j_s}\tau})\quad(\kappa_{j_s}\to\infty)\,,
$$
where $E(\cdot,\cdot)$ is the energy function introduced in \eqref{eq:E-FKP3D}.
\end{theorem}

Before presenting the proof of Theorem~\ref{thm:lb}, we discuss an easy consequence of it by taking $D=\Omega$ and $h=1$:

\begin{corollary}\label{corol:lb}
Under the assumptions in Theorem~\ref{thm:lb},
$$\mathcal G(\psi_{j_s},\nb_{j_s})\geq \sqrt{q_{j_s}\tau}\int_{\partial\Omega}
E\Big(\frac1b,\nu(x;\nb_0)\Big)\,ds(x)+ o(\sqrt{q_{j_s}\tau})\quad( \kappa_{j_s}\to\infty)\,.
$$
\end{corollary}

To prove the lower bound in Theorem~\ref{thm:lb}, we shall split the energy into two parts, the surface energy and the bulk energy, and they will be estimated separately.

\subsection{The field $\nb_0$}\label{sec:conv}

Here, we will construct the subsequence $\{\psi_{j_s},\nb_{j_s},\kappa_{j_s}\}$ and the vector field $\nb_0$ appearing in Theorem~\ref{thm:lb}.
This is the content of:
\begin{lemma}\label{lem:sum}
Under the assumptions in Theorem~\ref{thm:lb}, there exist
$\nb_0\in\mathcal C(\tau)$ and a subsequence $\{\nb_{j_s}\}$
such that,  as $s\to\infty$,
\begin{equation}\label{eq:conv*}
\nb_{j_s}\to\nb_0\quad{\rm in~}H^1_{\rm loc}\cap L^p \cap W^{1,r}(\Omega,\R^3),
\end{equation}
for $1\leq p<\infty$ and $1\leq r<2$, and also pointwise a.e. in $\Omega$.
Furthermore,  as $s\to\infty$,
\begin{equation}\label{lem:repair}
\nb_{j_s}\to \nb_0 ~{\rm uniformly~on~}\partial\Omega\,.
\end{equation}
\end{lemma}

Note that we work under the assumptions in Theorem~\ref{thm:lb}, in particular,
$$\nb_j\in\mathcal A\quad{\rm and}\quad \mathcal E(\psi_j,\nb_j)= \gsej\,,
$$
where the admissible class $\mathcal A$ is introduced in \eqref{eq:adm-n} and the energy $\gse$ is introduced in \eqref{eq-gs-LdeG}. By definition of the admissible class $\mathcal A$, there exists a sequence of vector fields, $\{\nb_{j}^{\partial\Omega}\}\subset\mathcal C(\tau)$ such that
\begin{equation}\label{eq:n=n0-bd}
\nb_{j}=\nb_{j}^{\partial\Omega}\quad{\rm on~}\partial\Omega\,.
\end{equation}
In the sequel, we will fix the choice of the sequence $\{\nb_j^{\partial\Omega}\}$ such that \eqref{eq:n=n0-bd} holds.

\begin{proof}[Proof of Lemma~\ref{lem:sum}]
Since $\nb_j\in\mathcal A$, we know that $\mathcal L(\nb_j)=0$. In light of \eqref{eq-LdeG*}, we get,
$$\mathcal E(\psi_j,\nb_j)=\mathcal G(\psi_j,\nb_j)+\mathcal F_N^+(\nb_j)
+\kappa_j^2\int_\Omega\Big\{{\rm tr}(D\nb_j)^2-|\Div\nb_j|^2\Big\}\,dx\,.
$$
We write a lower bound for $\mathcal E(\psi_j,\nb_j)$  using Lemma~\ref{lem:lb-OFen} and the inequality in \eqref{eq:lbGL},
$$\aligned
\mathcal E(\psi_j,\nb_j)\geq &\Big(\min(K_{1,j},K_{2,j},K_{3,j})-2\kappa_j^2\Big)
\int_\Omega\Big\{|\Div\nb_j|^2+|\curl\nb_j+\tau\nb_j|^2\Big\}\,dx\\
&\q +
\kappa_j^2\int_\Omega|D\nb_j|^2\,dx-
C|\Omega|\kappa_j^2\,.
\endaligned
$$
Note that, if $\nb_\tau\in\mathcal C(\tau)$, then it is in the admissible class $\mathcal A$ in \eqref{eq:adm-n} and
$$\mathcal E(\psi_j,\nb_j)\leq \mathcal E(0,\nb_\tau)=0\,.$$
Furthermore, using  Assumptions~\ref{assumption:A} and \ref{assumption:A'},  we see that $(\nb_j)_j$ satisfy for all $j$,
\begin{equation}\label{est:3.1*}
\|\Div\nb_j\|_2+\|\curl\nb_j+\tau\nb_j\|_2\leq C\kappa_j^{-1/2}|\ln\kappa_j|^{-1}=o(1)\quad{\rm and}\quad \|D\nb_j\|_2\leq C\,.
\end{equation}
We can now apply Lemma~\ref{lem:HP} to extract a vector field $\nb_0\in\mathcal C(\tau)$ and a subsequence $\kappa_{j_s}\to\infty$  (as $s\to\infty$) such that,
$\nb_{j_s}\to\nb_0$ with convergence in
$H^1_{\rm loc}(\Omega,\R^3)$, in
$L^p(\Omega;\R^3)$ for all $1\leq p<\infty$, and in
$W^{1,r}(\Omega;\R^3)$ for all $1\leq r<2$.
We can still refine the subsequence and get the additional convergence
$$
\nb_{j_s}\to\nb_0\quad{\rm in~}L^{r}(\partial\Omega;\R^3)\quad{\rm and}\quad \nb_{j_s}\to\nb_0~\text{\rm a.e. in }\Omega .
$$

We will refine the subsequence one more time to get that $\nb_{j_s}^{\partial\Omega}$ converges uniformly to $\nb_0$ on the boundary. Here $\{\nb_{j}\}$ is the sequence in $\mathcal C(\tau)$ and satisfying \eqref{eq:n=n0-bd}.  The class
$\mathcal C(\tau)$ is defined in \eqref{C(tau)}.

Since $\nb_{j_s}^{\partial\Omega}\in\mathcal C(\tau)$, $\nb_{j_s}^{\partial\Omega}(x)=Q_{j_s}N_\tau(Q_{j_s}^tx)$ where $Q_{j_s}\in SO(3)$ is an orthogonal matrix and $N_\tau(\cdot)$ is the smooth vector field introduced in \eqref{Ntau}.

By compactness of the orthogonal group $SO(3)$, we may extract a subsequence and an orthogonal matrix $Q_0$ such that $Q_{j_s}\to Q_0$ in the sense of matrices. Consequently, by defining the vector field
$$\nb_*(x)=Q_0N_\tau(Q_0^tx)\,,$$
we infer from the $L^r$ convergence and \eqref{eq:n=n0-bd}
$$\nb_*(x)=\nb_0(x)\quad{\rm on}~\partial\Omega\,.$$
The compactness of $\overline\Omega$ allows us to refine the subsequence further and get
$$\sup_{x\in\overline\Omega}|\nb_{j_s}^{\partial\Omega}(x)-\nb_*(x)|\to0\,.$$
Actually, by smoothness of $N_\tau(\cdot)$, we may define $x_{j_s}\in\overline\Omega$ such that
$$\sup_{x\in\overline\Omega}|\nb_{j_s}^{\partial\Omega}(x)-\nb_*(x)|=|\nb_{j_s}^{\partial\Omega}(x_{j_s})-\nb_*(x_{j_s})|
=|Q_{j_s}N_\tau(Q_{j_s}^tx_{j_s})-Q_0N_\tau(Q_0^tx_{j_s})|\,.$$
The compactness of  $\overline\Omega$  ensures that $x_{j_s}\to x_*$ along a subsequence, and \eqref{lem:repair} follows.
\end{proof}

{\bf In the rest of this section, we will skip the reference to the subsequence and write}
\eq\label{convention}
\boxed{
\kappa=\kappa_{j_s},\q
q=b\kappa_{j_s}^2/\tau,\q
(\psi,\nb)=(\psi_{j_s},\nb_{j_s})\,,\quad \nb^{\partial\Omega}=\nb_{j_s}^{\partial\Omega}\,.}
\eeq

We mention one more useful estimate:

\begin{lemma}\label{lem:n-n0:L1}
Using the convention  in \eqref{convention}, we have two  constants $C>0$ and $t_0>0$  such that, for all $t\in(0,t_0)$,
$$\|\nb-\nb^{\partial\Omega}\|_{L^2(\O_t)}\leq C \,t^{3/4}\,,
$$
where $\O_t$ is defined in \eqref{Omegat}.
\end{lemma}
\begin{proof}
Let $t_0>0$ be the constant in Lemma~\ref{lem:bndTrTh}. Our assumption on the vector fields $\nb$ and $\nb^{\partial\Omega}$ say that $|\nb|=|\nb^{\partial\Omega}|=1$ a.e. and $\nb=\nb^{\partial\Omega}$ along the boundary of $\Omega$. Consequently, for all $t\in(0,t_0)$,
$$
\|\nb-\nb^{\partial\Omega}\|^2_{L^2(\O_t)}\leq
2 \|\nb-\nb^{\partial\Omega}\|_{L^1(\O_t)}\leq Ct\|D\nb-D\nb^{\partial\Omega}\|_{L^1(\O_{2t})}\,.
$$
By compactness of the orthogonal group $SO(3)$, we get that $\|D\nb^{\partial\Omega}\|_\infty$ is bounded (recall that $\nb^{\partial\Omega}\in\mathcal C(\tau)$ and $\mathcal C(\tau)$ is defined in \eqref{C(tau)}).  We use the Cauchy-Schwarz inequality and the estimate in \eqref{est:3.1*} to write
$$
\|D\nb-D\nb^{\partial\Omega}\|_{L^1(\O_{2t})}\leq |\O_{2t}|^{1/2}\,\|D\nb-D\nb^{\partial\Omega}\|_{L^2(\Omega)}\leq Ct^{1/2}\,.
$$
\end{proof}

\subsection{Splitting into bulk and surface terms}
We will split the energy
\begin{equation}\label{eq:lb-E>E0}
\mathcal G(h\psi,\nb;D):=\int_D\left\{|(\nabla-iq\nb)\psi|^2-\kappa^2|\psi|^2+\frac{\kappa^2}2|\psi|^4\right\}\,dx\,.
\end{equation}
into  boundary and bulk parts.

We  introduce three parameters $\alpha, \delta, \varepsilon \in (0,1)$ depending on $\kappa$ that will be  used along the proof. Different conditions on these
parameters will arise so as to control the error terms correctly.
We work with the concrete choice.
\begin{align}\label{eq:ChoiceParams}
\delta = \kappa^{-1} |\ln \kappa|^{1/2}, \qquad \alpha = |\ln\kappa|^{-1/16}, \quad \varepsilon = |\ln\kappa|^{-1}
\end{align}
for concreteness. Notice in particular, that the parameter $\delta$ will have a different value than was the case for the upper bound.

We introduce smooth real-valued functions $\chi_1$ and $\chi_2$ such
that $\chi_1^2+\chi_2^2= 1$ in $\Omega$,
\[
\chi_1(x)=
\begin{cases}
1, &\text{if } \dist(x,\partial\Omega)<\delta/2,\\
0, &\text{if } \dist(x,\partial\Omega)>\delta,
\end{cases}
\]
and $|\nabla \chi_j|\leq C/\delta$ for $j=1,2$ and some constant $C$
independent of $\delta$.
Notice that $\chi_1 \nabla \chi_1 + \chi_2 \nabla \chi_2 =0$. Therefore, we get the localization formula,
$$\aligned
\mathcal G(h\psi,\nb;D)=&\mathcal G(\chi_1h\psi,\nb;D)+\mathcal G(\chi_2h\psi,\nb;D)\\
&+\sum_{j=1}^2\int_{D} (\chi_j(x)^2-\chi_j(x)^4)|h\psi|^4\,dx
-\sum_{j=1}^2\int_{D} |\nabla\chi_j|^2|h\psi|^2\,dx\,.
\endaligned
$$
We will use the facts that
\begin{align}\label{eq:Positivity}
\int_{D} (\chi_j(x)^2-\chi_j(x)^4)|h\psi|^4\,dx\geq
0,
\end{align}
since $0\leq\chi_j(x)\leq 1$,
that $\| h\psi
\|_{\infty} \leq 1$, and that the measure of the support of
$\nabla\chi_j$ is bounded by $C\delta$ for some constant $C$.
Therefore,
\begin{align}\label{eq-splitting}
\mathcal G(h\psi,\nb;D)
&\geq
\mathcal G(\chi_1h\psi,\nb;D)+\mathcal G(\chi_2h\psi,\nb;D)-C \delta^{-1}.
\end{align}
In the following  we estimate separately the terms $\mathcal
G(\chi_1h\psi,\nb;D)$ (surface energy) and\break
$\mathcal G(\chi_2h\psi,\nb;D)$ (bulk energy).

\subsection{The surface energy}

Let $\nb_0$ be the vector field constructed in Lemma~\ref{lem:sum} along with the two subsequences $\kappa=\kappa_{j_s}\to\infty$ (as $s\to\infty$) and $(\psi,\nb)=(\psi_{j_s},\nb_{j_s})$.  From here till the end of subsection 6.2, we use the convention in \eqref{convention} and write $``\kappa\to\infty"$ for $``\kappa_{j_s}\to\infty"$.
This section is devoted to the proof of
\begin{lemma}\label{Lem5.3}
The surface energy satisfies the following lower bound (as $\kappa \rightarrow \infty$)
\begin{equation}\label{eq:lower-bound}
\mathcal G(\chi_1h\psi,\nb;D) \geq
\sqrt{q\tau}\,\int_{\overline{D}\,\cap\partial\Omega}
E\Big(\frac{1}{b},\nu(x;\nb_0)\Big)\,ds(x)-o(\kappa)\,.
\end{equation}
\end{lemma}

\begin{proof}
The estimate of the surface energy requires two steps, a
decomposition of the energy via a partition of unity, then passing
to local boundary coordinates (introduced in
Section~\ref{sec:bndcod}) that allow us to compare with the reduced
Ginzburg-Landau energy (introduced in Section~\ref{sec:redGL}).

{\it Step 1}.
Let
$\alpha=\alpha(\kappa)\ll 1$ be a parameter that will be explicitly
chosen below. Let, for $\delta>0$,
\[
O_{\delta} := \bigl\{(y_1,y_2,y_3)\bigm| 0<y_1<\delta,\
-\delta <y_2<\delta,\ -\delta < y_3 < \delta\bigr\}\,.
\]
Consider the family $\{x_{0,l}\}_l\subset \partial\Omega$ introduced in Lemma~\ref{lem:p-unity}. For each
point $x_{0,l}$, we may  introduce a coordinate transformation
$\Phi_l$ valid near the point $x_{0,l}$ (see \eqref{eq:coo-tran}).
In light of Lemma~\ref{lem:p-unity}, we introduce a  partition of unity $\{\tchi_l\}_{l}$ covering the
set
$$\Omega_{\rm bnd}:=\supp\chi_1,
$$
such that
\begin{equation}\label{eq:choice-pts}
\left\{
\aligned
&\sum_l \tchi_l^2(x)= 1{~\rm and~} \tchi_l\geq 0\q\text{in }\Omega_{\rm bnd}\,,\\
&\tchi_l\equiv 1\q\text{in the set }Q_{\delta,l}:=\Phi_l^{-1}\bigl(O_{(1-\alpha)\delta}\bigr)\,,\\
&\exists~C>0,~\forall~x\in\Omega_1,~\sum_l |\nabla \tchi_l(x)|^2 \leq C(\alpha\delta)^{-2}\,.
\endaligned
\right.
\end{equation}
We write the following decomposition formula
$$
\aligned
\mathcal G(\chi_1h\psi,\nb;D)=&
\sum_{l} \biggl\{ \mathcal G(\tchi_l\,\chi_1h\psi,\nb;D)\\
&\qq +\int_{D} (\tchi_l(x)^2-\tchi_l(x)^4)|\chi_1h\psi|^4\,dx
-\int_{D\cap\Omega_1}|\nabla \tchi_l|^2|\chi_1h\psi|^2\, dx\biggr\}.
\endaligned
$$
By this and an inequality similar to \eqref{eq:Positivity}
we get the following lower bound of the surface energy,
$$
\mathcal G(\chi_1h\psi,\nb;D) \geq
\sum_{l} \biggl\{ \mathcal G(\tchi_l\,\chi_1h\psi,\nb;D)
-\int_{D\cap\Omega_{\rm bnd}}|\nabla \tchi_l|^2|\chi_1h\psi|^2\, dx\biggr\}.
$$
Using the bound on $\nabla\tchi_l$ from \eqref{eq:choice-pts},
the lower bound of the surface energy becomes
\begin{equation}\label{eq:n}
\mathcal G(\chi_1h\psi,\nb;D) \geq
\sum_{l} \mathcal G(\tchi_l\,\chi_1h\psi,\nb;D)
-C\alpha^{-2}\delta^{-1}.
\end{equation}

{\it Step 2}.
Now we derive estimates for the terms in the right side of \eqref{eq:n}.
In the following we estimate the terms of the form $\mathcal
G(\tchi_l\,\chi_1h\psi,\nb;D)$. As we shall see,
the approximation relies on the construction of a suitable gauge
transformation and using the local boundary coordinates.

The first step is to restrict the summation over the cells in $D$ by writing
\begin{equation}\label{eq:n*}
\mathcal G(\chi_1h\psi,\nb;D) \geq
\sum_{{\substack{l\\Q_{\delta,l}\subset D}}} \mathcal G(\tchi_l\,\chi_1h\psi,\nb;D)
-C\big(\delta^2\kappa^2+\alpha^{-2}\delta^{-1}\big)\,.
\end{equation}
The lower bound in \eqref{eq:n*} is a simple consequence of the fact that
\begin{equation}\label{eq:measure(D)}
\Big|(D\cap\Omega_{\rm bnd})\setminus\bigcup_{{\substack{l\\Q_{\delta,l}\subset D}}}\overline{Q_{\delta,l}}\Big|\leq C\delta^2\,,\end{equation}
which is a consequence of the assumption that the domain $D$ is regular (see Definition~\ref{def:dom}).

In light of \eqref{eq:n*}, we only deal with cells $Q_{\delta,l}\subset D$. In each cell $Q_{\delta,l}$, we will apply Lemma~\ref{lem:gt-S2} to obtain an estimate as in \eqref{eq:a=a-cst-int}. That way, we get a function $f_{0,l}\in H^1(Q_{\delta,l})$ such that the vector field $\nb$ satisfies
\begin{align}\label{eq:a=a-cst-int0}
\|\nb-\tau(\nb^{\partial\Omega})_{{\rm cst},l}&-\nabla f_{0,l}\|_{L^2(Q_{\delta,l})} \nonumber \\
\leq&
3\delta\sqrt{|\ln\delta|}\Big(\|\curl\nb+\tau\nb\|_{L^2(Q_{\delta,l})}+\tau\|\nb-\nb^{\partial\Omega}\|_{L^2(Q_{\delta,l})}\Big)
+C\delta^{3}\,.
\end{align}
Here, $\nb^{\partial\Omega}\in\mathcal C(\tau)$ is the field constructed satisfying \eqref{eq:n=n0-bd}  and $(\nb^{\partial\Omega})_{{\rm cst},l}$ is defined as follows,
\begin{equation}\label{eq:n0cst-l}
(\nb^{\partial\Omega})_{{\rm cst},l}=\int_0^1s(x-x_{0,l})\times\nb^{\partial\Omega}(x_{0,l})\,ds\,.
\end{equation}
Note that
\begin{equation}\label{eq:n0cst-l-curl}
\curl(\nb^{\partial\Omega})_{{\rm cst},l}=-\nb^{\partial\Omega}(x_{0,l})=\tau^{-1}\curl\nb^{\partial\Omega}(x_{0,l})\,.
\end{equation}
In light of the estimate in \eqref{eq:a=a-cst-int0} we have, for all $\varepsilon\in(0,1/2)$,
\eq\label{eq:n=n-avN}
\aligned
&\int_D|(\nabla-iq\nb)\widetilde\chi_l\chi_1h\psi|^2dx\\
\geq& (1-\varepsilon)\int_D|(\nabla-iq\tau (\nb^{\partial\Omega})_{{\rm cst},l})e^{-iqf_{0,l}}\widetilde\chi_l\chi_1h\psi|^2\,dx\\
&-C\varepsilon^{-1}q^2\delta^2|\ln\delta|\left(\|\curl\nb+\tau\nb\|_{L^2(Q_{\delta,l})}^2+\|\nb-\nb^{\partial\Omega}\|_{L^2(Q_{\delta,l})}^2\right)-C\varepsilon^{-1}q^2\delta^6\,.
\endaligned
\eeq
As indicated in \eqref{eq:n0cst-l}, the magnetic field  $\curl ( \nb^{\partial\Omega})_{{\rm cst},l}=-\nb^{\partial\Omega}(x_{0,l})$
is constant. At each point $x_{0,l}$,
 the director field $\nb^{\partial\Omega}(x_{0,l})$  forms an angle
\begin{equation}\label{eq:def-nu0l}
\nu_{0,l}=\nu(x_{0,l})=\arcsin(|\nb^{\partial\Omega}(x_{0,l})\cdot \Nb(x_{0,l})|)\in[0,\pi/2]
\end{equation}
with the tangent plane to $\partial\Omega$;
$\Nb(x)$ denotes the outward normal to $\partial\Omega$ at $x$.

We apply the result in Lemma~\ref{lem:gt2} to obtain a real
valued smooth function $\beta_{0,l}$ such that, if $\widetilde{(
\nb^{\partial\Omega})_{{\rm cst},l}}$ is the vector field defined in $y$-coordinates
by the relation in \eqref{eq-gauge-F} (with $\Fb=(\nb^{\partial\Omega})_{{\rm
cst},l}$), then 
it holds in $Q_{\delta,l}$,
\begin{equation}\label{eq:gauge**}
\Big|\widetilde{(\nb^{\partial\Omega})_{{\rm cst},l}}-\big(\Ab_{\nu_{0,l}}+\nabla\beta_{0,l}\big)\Big|\leq C\delta^2\,,
\end{equation}
where $\Ab_{\nu_{0,l}}$ is the magnetic potential from
\eqref{eq-3D-Eb}.
We introduce the function
\begin{equation}\label{eq-vl}
v_l=\tchi_l\,\chi_1h\psi\,\exp(-iqf_{0,l})\circ \Phi_l^{-1}\times \exp(-iq\beta_{0,l})\,,
\end{equation}
and the vector field $\Fb$ defined in $y$-coordinates by the relation $\widetilde \Fb=\widetilde{(\nb^{\partial\Omega})_{{\rm cst},l}}-\nabla\beta_{0,l}$. We mention again that for a given vector field $\Fb$ defined in the $x$-variable, we use $\widetilde \Fb$ to denote the new vector field in the $y$-variables corresponding to $\Fb$ via the formula \eqref{eq-gauge-F}. Let $(\widetilde F_1,\widetilde F_2,\widetilde F_3)$ be the components of the vector field $\tilde\Fb$ in $y$-coordinates. Also, let $\widetilde v_l$ be the function obtained from the function $v_l$ by converting to $y$-coordinates.
Using
\eqref{eq-bnd-en} we get,
\eq
\aligned
&\mathcal G(\tchi_l\,\chi_1v_l,({\nb^{\partial\Omega}})_{{\rm cst},l};D)\\
=&
\int_{O_{\delta}} \det(g_{jk})^{1/2}\biggl[
 \sum_{1\leq j,k\leq 3} g^{jk}
 \bigl(\partial_{y_j}-iq\tau \widetilde F_j\bigr)\widetilde v_l \times
  \overline{\bigl(\partial_{y_k}-iq\tau \widetilde F_k\bigr)\widetilde v_l}
  -\kappa^2 |\widetilde v_l|^2
  +\frac{\kappa^2}{2}|\widetilde v_l|^4 \biggr]\,dy\,.
\endaligned
\eeq
Inserting the estimates~\eqref{eq:metricapprox} and~\eqref{eq:jacobian} into the above equality we obtain (again it is assumed that $\delta$
is sufficiently small)
$$\aligned
&\mathcal G(\tchi_l\,\chi_1v_l,(\nb^{\partial\Omega})_{{\rm cst},l};D)\\
\geq&
(1-C\delta)\int_{O_{\delta}}\Big\{\big|\big(\nabla_y-iq\tau (\widetilde{(\nb^{\partial\Omega})_{{\rm cst},l}}-\nabla\beta_{0,l})\big)\widetilde v_l\big|^2
- \kappa^2|\widetilde v_l|^2+\frac{\kappa^2}{2}|\widetilde v_l|^4\Big\}\,dy\\
&-C\delta\kappa^2\int_{O_{\delta}}|\widetilde v_l|^2\,dx\,.
\endaligned
$$
To estimate the first term in the right side of the above inequality, we use \eqref{eq:gauge**} to write the pointwise inequality (with $\varepsilon\in(0,1/2)$ arbitrary)
\[
\big|\big(\nabla_y-iq\tau (\widetilde{(\nb^{\partial\Omega})_{{\rm cst},l}}-\nabla\beta_{0,l})\big)\widetilde v_l\big|^2
\geq (1-\varepsilon)|(\nabla_y-iq\tau\Ab_{\nu_{0,l}})\widetilde v_l|^2
+(1-\varepsilon^{-1})(q\tau)^2\delta^4\,.
\]
That way we infer from \eqref{eq:n=n-avN}
\begin{align*}
\mathcal G&(\tchi_l\,\chi_1\psi,\nb;D)\geq
(1-C\varepsilon-C\delta)\widetilde{\mathcal G}(\widetilde v_l,\Ab_{\nu_{0,l}})\\
&~
-C\varepsilon^{-1}(q\tau)^2\delta^4
\int_{O_{\delta}}|\widetilde v_l|^2 \,dy
-C(\varepsilon+\delta)\kappa^2\int_{O_{\delta}}|\widetilde v_l|^2\,dy\\
&~-C\varepsilon^{-1}q^2\delta^2|\ln\delta|\left(\|\curl\nb+\tau\nb\|_{L^2(Q_{\delta,l})}^2
+\|\nb-\nb^{\partial\Omega}\|^2_{L^2(Q_{\delta,l})}\right)-C\varepsilon^{-1}q^2\delta^6\,,
\end{align*}
where
\begin{equation}\label{eq:newG}
\widetilde{\mathcal G}(\widetilde v_l,\Ab_{\nu_{0,l}})
=\int_{O_{\delta}}\left\{|(\nabla_y-iq\tau \Ab_{\nu_{0,l}}(y))\widetilde v_l|^2
             - \kappa^2|\widetilde v_l|^2+\frac{\kappa^2}{2}|\widetilde v_l|^4\right\}\,dy\,.
\end{equation}
Recall that $\kappa^2/(q\tau)=\frac1b$ where $b$ is a constant in $(1,\infty)$.  As a consequence, we get,
\begin{align}\label{eq-lb-*}
\mathcal G(&\tchi_l\,\chi_1h\psi,\nb;D) \nonumber \\
&\geq
(1-C\varepsilon-C\delta)\widetilde{\mathcal G}(\widetilde v_l,\Ab_{\nu_{0,l}})
-C(\delta + \varepsilon+\varepsilon^{-1} \kappa^2\delta^4)\kappa^2 \int_{O_{\delta}}|v_l|^2\,dy\nonumber \\
&-C\varepsilon^{-1}\kappa^4\delta^2|\ln\delta|\left(\|\curl\nb+\tau\nb\|_{L^2(Q_{\delta,l})}^2
+\|\nb-\nb^{\partial\Omega}\|^2_{L^2(Q_{\delta,l})}\right)
-C\varepsilon^{-1}\kappa^4\delta^6\,.
\end{align}
After applying the  re-scaling $y=(q\tau)^{-1/2}z$, we obtain
\begin{equation}\label{eq-correction}
\widetilde{\mathcal G}(v_l,\Ab_{\nu_l})
=\frac{1}{\sqrt{q\tau}}\int_{O_{\sqrt{q\tau}\delta}}
  \left\{ |(\nabla_z-i\Ab_{\nu_l}(z))\widetilde v_l|^2
             - \frac1{b}|\widetilde v_l|^2+\frac{1}{2b}|\widetilde v_l|^4\right\}\,dz\,,
\end{equation}
where $\widetilde v_l(z)=v_l\big((q\tau)^{-1/2}z\big)$. Notice that our parameters satisfy
$\delta, \varepsilon \ll1$, $\sqrt{q\tau}\,\delta\gg1$.
By using
\eqref{eq-E(nu,b)} (with $\ell=\sqrt{q\tau}\,\delta$) we conclude
that,
\begin{align}
\label{eq:nn}
\widetilde{\mathcal G}(v_l,\Ab_{\nu_{0,l}})
\geq
(1-C\varepsilon-C\delta)\sqrt{q\tau}\,E\Big(\frac1b,\nu_{0,l}\Big)\,(4\delta^2)\,,
\end{align}
provided that $q$ is large enough (and
$\kappa^2/(q\tau)=\frac1b$).
We combine the estimates in \eqref{eq:n}-\eqref{eq:nn} to get,
\begin{align*}
&\mathcal G(\chi_1h\psi,\nb;D) \\
\geq& \sum_{\substack{l\\Q_{\delta,l}\subset D}}
(1-C\varepsilon-C\delta)\sqrt{q\tau}\,E\Big(\frac1b,\nu_{0,l}\Big)(4\delta^2)
-C(\delta +\varepsilon+\varepsilon^{-1}\kappa^2\delta^4 ) \kappa^2\sum_l \int_{O_{\delta}}|\widetilde v_l|^2\, dy\\
&\quad -C\varepsilon^{-1}\kappa^4\delta^2|\ln\delta|\left(\|\curl\nb+\tau\nb\|_{L^2(\Omega)}^2+\|\nb-\nb^{\partial\Omega}\|^2_{L^2(\{{\rm dist}(x,\partial\Omega)\leq C\delta\} )}\right)\\
&\quad -C\varepsilon^{-1}\kappa^4\delta^4- C\alpha^{-2}\delta^{-1}.
\end{align*}
To control the error terms in the right side of the above inequality, we estimate as before using the finite overlap of the supports of the partition of unity, and we get
\begin{align}
\sum_{l} \int_{O_{\delta}}|v_l|^2\, dy \leq C \int_{\Omega_{\rm bnd}} |v|^2 \,dx =C \int_{\Omega_{\rm bnd}} |\psi|^2 \,dx \leq C \delta.
\end{align}
The terms $\|\curl\nb+\tau\nb\|_2$  is estimated
in  \eqref{est:3.1*}. The term $\|\nb-\nb^{\partial\Omega}\|_{L^2(\{{\rm
dist}(x,\partial\Omega)\leq C\delta\}}$ is estimated by
Lemma~\ref{lem:n-n0:L1}.
That way we get,
\begin{align}\label{eq:finalF*}
\mathcal G(\chi_1h\psi,\nb;D)&
\geq \sum_{\substack{l\\Q_{\delta,l}\subset D}}
(1-C\varepsilon-C\delta)\sqrt{q\tau}\,E\Big(\frac1b,\nu_{0,l}\Big)(4\delta^2) -C(\delta +\varepsilon+ \varepsilon^{-1}\kappa^2\delta^4) \delta\kappa^2 \nonumber
\\
& -C\varepsilon^{-1}\kappa^4\delta^2|\ln\delta|\left(\kappa^{-1}|\ln\kappa|^{-2}o(1)+\delta^{3/2}\right) -C\varepsilon^{-1}\kappa^4\delta^4- C\alpha^{-2}\delta^{-1}.
\end{align}

With the choice \eqref{eq:ChoiceParams} of the parameters we find that all the error terms in \eqref{eq:finalF*} are of order
$o(\kappa)$ when $\kappa\to\infty$.
Thus,
\begin{equation}\label{eq:finalF**}
\mathcal G(\chi_1h\psi,\nb;D)
\geq \sum_{l*}(1-C\varepsilon-C\delta)\sqrt{q\tau}\,E\Big(\frac1b,\nu_{0,l}\Big)(4\delta^2)
-o(\kappa).
\end{equation}
where $\sum_{l*}$ is the sum over all $l$ such that $Q_{\delta,l}\subset D$.

Now we estimate the sum in \eqref{eq:finalF**}.  Recall the definition of the angle $\nu_{0,l}$ in \eqref{eq:def-nu0l}. By Lemma~\ref{lem:repair}, we get that,  as $\kappa\to\infty$,
$$
E\Big(\frac1b,\nu_{0,l}\Big)-E\Big(\frac1b,\nu(x_{0,l};\nb_0)\Big)\to0\quad\text{\rm uniformly in } l\,.
$$
Here $\nb_0$ is the vector field constructed in Sec.~\ref{sec:conv} and satisfying \eqref{eq:conv*}. The angle $\nu(x_{0,l};\nb_0)\in[0,\pi/2]$ is
$$\arcsin(|\nb_0(x_{0,l})\cdot \Nb(x_{0,l})|),$$
where $\Nb$ is the unit outward normal vector on $\partial\Omega$.

That way, the sum in \eqref{eq:finalF**}  can be estimated as follows
$$\sum_{l*}(1-C\varepsilon-C\delta)\sqrt{q\tau}\,E\Big(\frac1b,\nu_{0,l}\Big)(4\delta^2)=
\sum_{l*}(1-C\varepsilon-C\delta)\sqrt{q\tau}\,E\Big(\frac1b,\nu(x_{0,l};\nb_0)\Big)(4\delta^2) +o(1)\,.$$
The sum on the right side can be estimated
via a Riemann sum of the continuous function $\partial\Omega\ni x\mapsto E\big(\frac1b,\nu(x;\nb_0)\big)$. That way we obtain  that, as $\kappa\to\infty$,
\begin{equation}\label{eq:RS*}
\sum_{l*}
\Bigl\{ E\Big(\frac1b,\nu_{0,l}\Big)(4\delta^2)\Bigr\}
=\int_{\overline{D}\cap\partial\Omega} E\Big(\frac1b,\nu(x;\nb_0)\Big)\,ds(x)+o(1)\,.
\end{equation}
We insert \eqref{eq:RS*} into \eqref{eq:finalF**}, to get \eqref{eq:lower-bound}.
\end{proof}

\subsection{The bulk energy}

Compared to the estimate of the surface energy, the estimation of the term $\mathcal
G(\chi_2h\psi,\nb;D)$ is easy.
Recall the choice of parameters \eqref{eq:ChoiceParams} involved in the definition of $\chi_2$.
We will prove that

\begin{lemma}\label{lem:ub:op}
Under the assumptions in Theorem~\ref{thm:lb}, the subsequence in Lemma~\ref{Lem5.3} $(\psi,\nb,\kappa)=(\psi_{j_s},\nb_{j_s},\kappa_{j_s})$ satisfies,
\begin{equation}\label{eq-ub:op}
\int_{\{{\rm dist}(x,\partial\Omega)\geq\kappa^{-1}|\ln\kappa|^{1/2}\}}|\psi|^2\,dx=o(\kappa^{-1} )\,,
\end{equation}
and
\begin{equation}\label{eq-lb-bulk-conc}
\mathcal G(\chi_2h\psi,\nb;D)\geq-o(\kappa )\qquad \text{\rm as }\kappa\to\infty\,.
\end{equation}
\end{lemma}
\begin{proof} {\it Step 1.}
The first step in the proof of Lemma~\ref{lem:ub:op} consists of
determining a lower bound to $\mathcal G(\chi_2h\psi,\nb)$. Let $\delta$, $\varepsilon$ and $\alpha$ be the positive parameters introduced in \eqref{eq:ChoiceParams}.
We cover $\R^3$ by cubes $\{Q_1(x_{j,\alpha})\}_j$, where for all
$j\in\mathbb Z^3$, $y=(y_1,y_2,y_3)\in\R^3$ and $\delta>0$, we define,
$$x_{j,\alpha}=(1-\alpha)j\,,\quad Q_\ell(y)=\prod_{k=1}^3\left(y_k-\frac\delta2,y_k+\frac\delta2\right)\,.$$
Let $(g_j)$ be a partition of unity in $\R^3$ such that,
$$\sum_j g_j^2=1\,,\quad {\rm supp}\,g_j\subset Q_1(x_{j,\alpha})\,,\quad |\nabla g_j|\leq \frac{C}{\alpha}\,,$$
for some universal constant $C$.

Defining the re-scaled functions,
$$\forall~x\in\R^3\,,\quad g_{j,\delta}(x)=g_j(x/\delta)\,,$$
we get a new partition of unity $(g_{j,\delta})_j$ such that each
$g_{j,\delta}$ has support in the cube $Q_{j,\delta}:=Q_\delta(x_{j,\alpha})$  of side length $\delta$ and,
$$\sum_j g_ {j,\delta}^2=1\,,\quad |\nabla g_{j,\delta}|\leq\frac{C}{\alpha\delta}\,.$$
Let
$$\mathcal J=\{j\in\mathbb Z^3~:~{\rm supp}\,\chi_2\cap {\rm
supp}\,g_{j,\delta}\not=\emptyset\},\quad \text{and}\quad
\mathcal N_\delta={\rm
Card}\,\mathcal J.
$$
Then we know that, for $\delta$ and $\alpha$ sufficiently small,
\eq
|\Omega|\leq \mathcal N_\ell\times \delta^3\leq |\Omega|+O(\delta)+O(\alpha)\,.
\eeq
We have the localization formula, with arguments as before,
\begin{align}\label{eq-lb-mET}
\mathcal G(\chi_2h\psi,\nb)&\geq
\sum_{j\in\mathcal J} \mathcal G(g_{j,\delta}\chi_2h\psi,\nb)-\sum_j\int_\Omega |\nabla g_{j,\delta}|^2|\chi_2h\psi |^2\,dx\nonumber \\
&\geq
\sum_{j\in\mathcal J} \mathcal G(g_{j,\delta}\chi_2h\psi,\nb)-\frac{C}{\alpha^2\delta^2}\int_\Omega|h\psi|^2\,dx\,,
\end{align}
using that $\chi_2 \leq 1$ and the finite overlap of the $g_{j,\delta}$'s.

Now we estimate the first term in the right side of \eqref{eq-lb-mET}. We define the constant vector field
$$\nb_{{\rm av},j}=\frac1{|Q_{j,\delta}|}\int_{Q_{j,\delta}}\nb\,dx\,,$$
along with magnetic potential
$$\mathbf a_{{\rm av},j}=-\int_0^1s(x-x_{j,\delta})\times \nb_{{\rm av},j}\,ds\,.$$
We apply Lemma~\ref{lem:gt-S2*}  in $Q_{j,\delta}$ to obtain a function $f_{j,\delta}$ such that,
$$\|\nb-\tau\mathbf a_{{\rm av},j}-\nabla f_{j,\delta}\|^2_{L^2(Q_{j,\delta})}\leq C\delta^2|\ln\delta|\Big(\|\curl\nb+\tau\nb\|^2_{L^2(Q_{j,\delta})}+\delta^2\|D\nb\|^2_{L^2(Q_{j,\delta})}\Big)+C\delta^9\,.$$
It is therefore possible to estimate the `kinetic energy' term in $\mathcal
G(g_{j,\delta}\chi_2h\psi,\nb)$ from below as
follows:
\begin{align}\label{eq-lb-'}
\int_\Omega|&(\nabla-iq\nb)g_{j,\delta}\chi_2\psi|^2\,dx
\geq
(1-\varepsilon)\int_\Omega|(\nabla-iq\tau\mathbf a_{{\rm av},j})e^{-q f_{j,\delta}}g_{j,\delta}\chi_2\psi|^2\,dx
\nonumber\\
&\quad-C\varepsilon^{-1}q^2\delta^2|\ln\delta|\Big(\|\curl\nb+\tau\nb\|^2_{L^2(Q_{j,\delta})}+\delta^2\|D\nb\|^2_{L^2(Q_{j,\delta})}\Big)+C\varepsilon^{-1}q^2\delta^9\,.
\end{align}
Here $\varepsilon$ is an arbitrary real number in $(0,1)$ whose
choice will be specified later.

Referring to \eqref{est:3.1*}, we know that $\|D\nb\|_{L^2(Q_{j,\delta})}\leq C$ where $C$ is a constant independent of $\delta$, $\kappa$ and $q$.
 By Lemma~\ref{lem:gt-S2*},
$\curl\mathbf a_{{\rm av},j}$ is constant
and has magnitude $\geq 1-C\ell$. Since  every function $g_{j,\delta}\chi_2\psi$ has
compact support in $\Omega$,  we may write using \eqref{eq:spect-est},
\begin{equation}\label{eq-lb-linear}
\int_\Omega|(\nabla-iq\tau\mathbf a_{{\rm av},j})e^{-iqf_{j,\delta}}g_{j,\delta}\chi_2\psi|^2\,dx\geq q\tau(1-C\delta)\int_\Omega|g_{j,\delta}\chi_2\psi|^2\,dx\,.
\end{equation}
Inserting \eqref{eq-lb-linear}, \eqref{eq-lb-'} into
\eqref{eq-lb-mET}, we get,
\eq\label{eq-lb-conc'}
\aligned
\mathcal G(\chi_2h\psi,\nb)\geq& \Big[q\tau(1-C\delta)(1-\varepsilon)-\kappa^2-\frac{C}{\alpha^2\delta^2}\Big]\int_\Omega|\chi_2\psi|^2\,dx\\
&-C\varepsilon^{-1}q^2\delta^2|\ln\ell|\Big(\|\curl\nb+\tau\nb\|^2_{L^2(\Omega)}+\delta^2\|D\nb\|^2_{L^2(\Omega)}\Big)+C\varepsilon^{-1}q^2\delta^6\,.
\endaligned
\eeq
We estimate the error term using the inequalities in \eqref{est:3.1*} to obtain,
\eq\label{eq-lb-conc''}
\aligned
\mathcal G(\chi_2h\psi,\nb)\geq&
\Big[q\tau(1-\varepsilon)(1-C\delta)-\kappa^2-\frac{C}{\alpha^2\delta^2}\Big]
\int_\Omega|\chi_2\psi|^2\,dx\\
&-C\varepsilon^{-1}q^2\delta^2|\ln\ell|
\left(\kappa^{-1}|\ln\kappa|^{-2}o(1)+\delta^2\right)-C\varepsilon^{-1}q^2\delta^6\,.
\endaligned
\eeq
By assumption, the parameters $q$, $\tau$ and $\kappa$ satisfy
$$q\tau=b\kappa^2\,,\quad
b>1\,.$$ Consequently, when $\kappa$ is
sufficiently large, the choice of the parameters $\delta$, $\alpha$, $\varepsilon$ in \eqref{eq:ChoiceParams} yields
$$q\tau\big(1-\varepsilon)(1-C\delta)-\kappa^2-\frac{C}{\alpha^2\delta^2}\geq\left(b-1\right)\kappa^2-\kappa^2o(1)\geq \frac{(b-1)}2\kappa^2>0\,,
$$
and
$$\varepsilon^{-1}q^2\delta^2|\ln\delta|
\left(\kappa^{-1}|\ln\kappa|^{-2}o(1)+\delta^2\right)+\varepsilon^{-1}q^2\delta^6=o(\kappa )\,.
$$
Inserting this into \eqref{eq-lb-conc''}, we get,  as $\kappa\to\infty$,
\begin{equation}\label{eq:lb:en:op}
\mathcal G(\chi_2h\psi,\nb)\geq
\frac{(b-1)}{2}\kappa^2\int_\Omega|\chi_2\psi|^2\,dx+o(\kappa )\,.
\end{equation}

{\it Step 2}. Now, we insert \eqref{eq:lb:en:op} and the
lower bound in Lemma~\ref{eq:lower-bound} (used with $D=\Omega$)
into \eqref{eq-splitting}. We get that,
$$\mathcal G(\psi,\nb)\geq \sqrt{q\tau}\int_{\partial\Omega}E\big(\bb,\nu(x;\nb_0)\big)\,ds(x)+
\frac{(b-1)}{2}\kappa^2\int_\Omega|\chi_2\psi|^2\,dx+o(\kappa)\,.
$$
We insert the upper bound for $\mathcal G(\psi,\nb)$ given
in Proposition~\ref{prop:ub}, then we cancel the matching terms in
the resulting inequality to  get\,,
\begin{equation}\label{eq:ub:op}
\int_\Omega|\chi_2\psi|^2\,dx\leq o(\kappa^{-1})\,.
\end{equation}
Dropping the positive terms in the energy $\mathcal
G(\chi_2h\psi,\nb;D)$ and using \eqref{eq:ub:op} to estimate the
negative term, we get the inequality in \eqref{eq-lb-bulk-conc}.

The  estimate in \eqref{eq-ub:op} results from \eqref{eq:ub:op} by
recalling that $\chi_2=1$ in $\{{\rm
dist}(x,\partial\Omega)\geq\delta\}$ and that $\delta=|\ln\kappa|^{1/2}\kappa^{-1}$.
\end{proof}

\subsection{Proof of Theorem~\ref{thm:lb}}

The proof of Theorem~\ref{thm:lb} follows by collecting the estimates in \eqref{eq-splitting} and Lemmas~\ref{Lem5.3} and \ref{lem:ub:op}.\qed

\section{Proof of the main theorems}
In this section, we will present the proof of Theorems~\ref{thm:en*} and~\ref{thm:op}.

\subsection{Proof of Theorem~\ref{thm:en*}}

Proposition~\ref{prop:ub} yields
$$\limsup_{\kappa\to\infty}\frac{\gse}{\sqrt{q\tau}}\leq\Eground\Big(\frac1b,\tau\Big)\,.
$$
Suppose the conclusion in Theorem~\ref{thm:en*}
were false, then there exists a sequence $\kappa_j\to\infty$ such that
$$\lim_{j\to\infty}\frac{\gsej}{\sqrt{q\tau}}<\Eground\Big(\frac1b,\tau\Big)\,.
$$
Let $(\psi_j,\nb_j)$ be a minimizer of the energy in \eqref{eq-gs-LdeG} with $\k=\k_j$ and $q=q_j$. By the assumption in Theorem~\ref{thm:en*}, we know that
$$\mathcal G(\psi_j,\nb_j)\leq \mathcal E(\psi_j,\nb_j)=\gsej\,.
$$
Owing to Corollary~\ref{corol:lb}, there exists a subsequence $\kappa_{j_s}\to\infty$ and a vector field $\nb_0\in\mathcal C(\tau)$ such that
$$\int_{\partial\Omega}E\Big(\frac1b,\nu(x,\nb_0)\Big)\,ds(x)\leq \lim_{s\to\infty}\frac{\gsejs}{\sqrt{q\tau}}<\Eground\Big(\frac1b,\tau\Big)\,.
$$
This violates the definition of $\Eground$ in \eqref{eq-F0(nb)}.\qed

\subsection{Proof of Theorem~\ref{thm:op}}

The first assertion is proved in Sec.~\ref{sec:conv}, namely in \eqref{eq:conv*}. Using the inequality,
\begin{equation}\label{eq:G<E**}
\mathcal G(\psi_{j_s},\nb_{j_s})\leq \mathcal E(\psi_{j_s},\nb_{j_s})\,,\end{equation}
we infer from Theorem~\ref{thm:en*} and Corollary~\ref{corol:lb} that the vector field
$\nb_0$ satisfies the additional property $\nb_0\in\mathcal M$, where $\mathcal M$ is the set introduced in \eqref{eq:M(tau)}.

The second assertion in Theorem~\ref{thm:op} is obtained by combining the results in Theorems~\ref{thm:en*} and \ref{thm:lb} (and Corollary~\ref{corol:lb}).

The rest of this section is devoted to the lengthy proof of the
last assertions in Theorem~\ref{thm:op}. We will split the proof
into several steps.  In the sequel,  following the convention in \eqref{convention}, $(\psi,\nb)=(\psi_{j_s},\nb_{j_s})$ is the subsequence of  minimizers of
the functional in \eqref{eq-LdeG}, and $D\subset\Omega$ is a regular
domain as in Definition~\ref{def:dom}.

\subsubsection*{Step~1. Local energy estimate}
We have the simple decomposition of the energy in \eqref{eq:GL},
\begin{equation}\label{eq:decomp-D-Dc*}
\mathcal G(\psi,\nb)=\mathcal G(\psi,\nb;D)+\mathcal G(\psi,\nb;\overline{D}^c)\,,
\end{equation}
where
$\overline{D}^c=\Omega\setminus\overline{D}$.

Using Theorem~\ref{thm:lb} for the domain $\overline{D}^c$ and with $h=1$, we get,
\begin{equation}\label{eq:thm:lb:Dc}
\mathcal G(\psi,\nb;\overline{D}^c)\geq \sqrt{q\tau}\,\int_{\overline{D^c}\cap\partial\Omega}E\big(\frac{1}{b},\nu(x;\nb_0)\big)\,d s(x)+o(\kappa )\,.
\end{equation}

In light of the inequality in \eqref{eq:G<E**} and the result in Proposition~\ref{prop:ub}, we get an upper bound for $\mathcal G(\psi,\nb)$. Using this and the lower bound for $\mathcal G(\psi,\nb;\overline{D}^c)$ in \eqref{eq:thm:lb:Dc}, we infer from \eqref{eq:decomp-D-Dc*},
$$\mathcal G(\psi,\nb;D)\leq \sqrt{q\tau}\,\int_{\overline{D}\cap\partial\Omega}E\big(\frac{1}{b},\nu(x;\nb_0)\big)\,d s(x)+o(\kappa)\,.
$$
Combining this upper bound and the lower bound \eqref{eq:thm:lb:Dc},
we get the asymptotics in \eqref{eq:locEnergy}.

\subsubsection*{Step~2. Global estimate of the $L^4$-norm}
The order parameter $\psi$ satisfies the equation
$$-(\nabla-iq\nb)^2\psi=\kappa^2(1-|\psi|^2)\psi\quad {\rm in~}\Omega\,,
$$
with magnetic Neumann boundary conditions on $\partial\Omega$. Multiplying both sides of the equation by $\overline{\psi}$ and integrating by parts, we get,
$$\frac{\kappa^2}2\int_\Omega|\psi|^4\,dx=-\mathcal G(\psi,\nb)\,.
$$
Using the estimate for $\mathcal G(\psi,\nb)$ in \eqref{eq:locEnergy} (with $D=\Omega$), we get,
\begin{equation}\label{eq:op-Omega}
\kappa\int_\Omega|\psi|^4\,dx=-2\frac{\sqrt{q\tau}}{\kappa}\int_{\partial\Omega}E\big(\frac{1}{b},\nu(x;\nb_0)\big)\,d s(x)+o(1)\,.
\end{equation}

\subsubsection*{Step~3. Local estimate of the $L^4$ norm}
Here we use a method in \cite{HK} to determine an upper bound for
$\displaystyle\int_D|\psi|^4\,dx$ for any subdomain $D$ of $\O$. This is a new ingredient in the
proof compared to the similar statements for the Ginzburg-Landau
order parameter in \cite{FKP3D}. The novelty is that we do not use
elliptic {\it a priori} estimates satisfied by $\psi$.

By the assumption that the domain $D$ is regular in the sense
described in Definition~\ref{def:dom}, we may reduce the analysis to
 distinguish between the case when $\overline{D}\cap\partial\Omega$  is empty and
the case when $\overline{D}\cap\partial\Omega$ is a smooth surface without boundary
(or a smooth surface with a piece-wise smooth boundary). In general,
$D$ will be the union of a finite number of sub-domains
satisfying the aforementioned assumptions.

{\it Case~1. $\overline{D}\cap\partial\Omega=\emptyset$.}

In this case, the integral on $\partial\Omega\cap\overline{D}$
vanishes. So we have to prove that
$$\kappa\int_D|\psi|^4\,dx=o(1)\quad(\text{\rm as } \kappa\to\infty)\,.$$
Since $\overline{D}\subset\Omega$, then for $\kappa$ sufficiently large,
$D\subset\{{\rm dist}(x,\partial\Omega)\geq\kappa^{-1/2}\}$.
That way, in virtue  of  \eqref{eq-ub:op}, we write,
$$
\int_{D}|\psi|^4\,dx\leq
\int_{\{{\rm dist}(x,\partial\Omega)\geq\kappa^{-1/2}\}}|\psi|^4\,dx
=o(\kappa^{-1})\,.
$$

{\it Case~2. $\overline{D}\cap\partial\Omega$ is a smooth surface
(without boundary or with a piece-wise smooth boundary).}

In this case, the assumption on $D=\widetilde D\cap\Omega$ guarantees that $\overline D\cap\partial\Omega$ is a finite union of smooth surfaces in $\partial\Omega$.

Let
$$\ell=\kappa^{-1}|\ln\kappa|\quad{\rm and}\quad D_\ell=\{x\in
D~:~{\rm dist}(x,{\partial D})\geq \ell\}\,.$$ Consider a cut-off
function $\chi_{\ell}\in C_0^\infty(D)$ such that,
$$\|\chi_\ell\|_\infty\leq1\,,\quad \|\nabla\chi_\ell\|_\infty\leq
\frac{C}{\ell}\,,\quad\chi_\ell =1{~\rm in~}D_\ell\,,\quad{\rm and}\quad {\rm supp}\,\chi_\ell\subset \overline{D_{\ell/2}}\,.$$
The function $\chi_\ell^2\psi$ is easily seen to belong to $H^1_0(D)$.
Multiplying
both sides of the equation in \eqref{eq:GL} by
$\chi_\ell^2\overline\psi$ then integrating over $D$, we get,
\begin{equation}\label{eq:grad0}
\int_D\left\{|(\nabla-iq
\nb)\chi_\ell\psi|^2-\kappa^2\chi_\ell^2|\psi|^2+\kappa^2\chi_\ell^2|\psi|^4\right\}\,dx
=\int_D|\nabla\chi_\ell|^2|\psi|^2\,dx\,.
\end{equation}
We estimate using the bounds $\|\psi\|_\infty\leq 1$, $|\nabla\chi_\ell|\leq C/\ell$ and
the condition on the support of $\nabla\chi_\ell$
\begin{align*}
\int_D|\nabla\chi_\ell|^2|\psi|^2\,dx\leq C \ell^{-2} \int_{\{\dist(x, \partial D) \leq \ell\}}\,dx = C \ell^{-1} = o(\kappa).
\end{align*}
That way, we infer from \eqref{eq:grad0}  that, as $\kappa\to\infty$,
\begin{equation}\label{eq:grad0'}
\int_D\{|(\nabla-iq
\nb)\chi_\ell\psi|^2-\kappa^2\chi_\ell^2|\psi|^2+\kappa^2\chi_\ell^2|\psi|^4\}\,dx=o(\kappa)\,.
\end{equation}
We estimate $\displaystyle\int_D\chi_\ell^2|\psi|^4\,dx$ as follows. We write
\begin{equation}\label{eq:grad1}
\int_D\chi_\ell^2|\psi|^4\,dx=\int_{D} |\psi|^4\,dx+\int_D(1-\chi_\ell^2)|\psi|^4\,dx
\,.
\end{equation}
Since $\ell=\kappa^{-1}|\ln\kappa|\geq \kappa^{-1}|\ln\kappa|^{1/2}$, we may use the estimate in \eqref{eq-ub:op} to write
\begin{equation}\label{eq:grad1'}
\int_{D\cap\{{\rm dist}(x,\partial\Omega)\geq \ell\}}
(1-\chi_\ell^2)|\psi|^4\,dx=o(\kappa^{-1})\,.
\end{equation}
The condition on the support of $1-\chi_\ell$ and the assumption on the regularity of $D$ imply that (see Definition~\ref{def:dom})
$$\big|D\cap({\rm supp}\big(1-\chi_\ell)\big)\cap\{{\rm dist}(x,\partial\Omega)\leq \ell\}\big|=O(\ell^2)=o(\kappa^{-1})\,,$$
and consequently
$$ \int_{D\cap\{{\rm dist}(x,\partial\Omega)\leq \ell\}}
(1-\chi_\ell^2)|\psi|^4\,dx=o(\kappa^{-1})\,.$$
Inserting this estimate and \eqref{eq:grad1'} into \eqref{eq:grad1}, we get
\begin{equation}\label{eq:grad0''}
\int_D\chi_\ell^2|\psi|^4\,dx=
\int_D|\psi|^4\,dx+o(\kappa^{-1})\,.
\end{equation}
Since $1\geq \chi_\ell^2\geq \chi_\ell^4$, the estimates \eqref{eq:grad0'} and \eqref{eq:grad0''}
imply,
\begin{align}\label{eq:op-loc'}
-\frac{\kappa^2}{2}\int_D\chi_\ell^2|\psi|^4\,dx&\geq \mathcal
G(\chi_\ell\psi,\nb;D)-o(\kappa)\nonumber \\
&\geq
\sqrt{q\tau}\,\int_{\overline{D}\cap\partial\Omega}E\big(\frac{1}{b},\nu(x;\nb_0)\big)\,d s(x)-o(\kappa)\,,
\end{align}
where we used that Theorem~\ref{thm:lb} and that $\|\chi_\ell\|_\infty\leq 1$.
That way, we infer from \eqref{eq:grad0''} and \eqref{eq:op-loc'}  that, as $\kappa\to\infty$,
\begin{equation}\label{eq:op-up}
\int_{D}|\psi(x)|^4\,dx\leq -\frac{2\sqrt{q\tau}}{\kappa^2}\int_{\overline{D}\cap\partial\Omega}E\big(\frac{1}{b},\nu(x;\nb_0)\big)\,d s(x)+o(\kappa^{-1}).
\end{equation}

\subsubsection*{Step 4: Lower bound}
Notice that \eqref{eq:op-up} is valid when $D$
is replaced by the complementary of $\overline{D}$ in $\Omega$, i.e.
$\overline{D}^c$. We have the simple decomposition,
\begin{align*}
\int_{D}|\psi(x)|^4\,dx&=\int_{\Omega}|\psi(x)|^4\,dx-\int_{\overline{D}^c}|\psi(x)|^4\,dx\\
&\geq
\int_{\Omega}|\psi(x)|^4\,dx-\frac{2\sqrt{q\tau}}{\kappa^2}\int_{\overline{D^c}\cap\partial\Omega}E\big(\frac{1}{b},\nu(x;\nb_0)\big)\,d s(x)+o(\kappa^{-1})\,.
\end{align*}
Using the asymptotics in \eqref{eq:op-Omega} obtained in  Step~2, we
deduce that,  as $\kappa\to\infty$,
$$\int_{D}|\psi(x)|^4\,dx\geq -\frac{2\sqrt{q\tau}}{\kappa^2}\int_{\overline{D}\cap\partial\Omega}E\big(\frac{1}{b},\nu(x;\nb_0)\big)\,d s(x)+o(\kappa^{-1})\,.
$$
Combining this lower bound and the upper bound in \eqref{eq:op-up},
we obtain the asymptotics announced in the fourth assertion of
Theorem~\ref{thm:op}.
This finishes the proof of Theorem~\ref{thm:op}.\qed

\subsection{Proof of Theorem~\ref{thm:gse}}

Here we present the proof of Theorem~\ref{thm:gse} which
is relatively easier than that of Theorem~\ref{thm:op}. The proof is
along similar calculations done in \cite{FKP3D}, so we will be
rather succinct here.

\subsubsection{Upper bound}
Recall that $b>1$ is a fixed constant and $q\tau=b\kappa^2$.  Since $\mathcal C(\tau)$ is a closed
set in a finite dimensional space, then we can choose $\nb_0 \in
{\mathcal C}(\tau)$ such that
$$\tilde E\Big(\frac1b,\nb_0\Big)=\Eground\Big(\frac1b,\tau\Big)\,,$$
where $\tilde E$ and $\Eground$ are the energies in
\eqref{E(bnu)} and \eqref{eq-F0(nb)}.

We have, for all $\psi\in H^1(\Omega;\mathbb C)$
\begin{align}
{\mathcal E}(\psi, \nb_0) =\mathcal G(\psi,\nb_0)=
 \int_{\Omega}\Big\{ |\nabla_{q \nb_0} \psi|^2 - \kappa^2|\psi|^2 + \frac{\kappa^2}{2} |\psi|^4\}\, dx.
\end{align}
Upon minimizing this energy over $\psi$, we will get an upper bound
to the full energy in \eqref{eq-LdeG}. This will be done through the
computation of the energy of a relevant  test configuration, whose
construction hints at the expected behavior of the actual minimizers
of the energy in \eqref{eq-LdeG}. We take the same trial function $\psi_{\rm trial}$ in the proof of Lemma~\ref{lem:ub**}. Thanks to Lemma~\ref{lem:ub**}, we get, for all $K_1,K_2,K_3\geq 0$,
$$
\gse
\leq
(1+C\eta)\sqrt{q\tau}\Eground\left(\frac1b,\tau\right)
+C\eta+\kappa f(\kappa)\,,
$$
where the function $f(\cdot)$ is independent of $K_1,K_2,K_3$ and $f(\kappa)=o(1)$ as $\kappa\to\infty$.
Taking  successively the limit as $\min(K_1,K_2,K_3)\to\infty$ then as $\kappa\to\infty$, we get,
$$\limsup_{\kappa\to\infty}\Big(\limsup_{\min(K_1,K_2,K_3)\to\infty}\frac{\gse}{\sqrt{q\tau}}\Big)\leq
(1+C\eta)\sqrt{q\tau}\Eground\Big(\frac1b,\tau\Big)+C\eta\,.$$
Sending $\eta\to0_+$, we get
$$\limsup_{\kappa\to\infty}\Big(\limsup_{\min(K_1,K_2,K_3)\to\infty}\frac{\gse}{\sqrt{q\tau}}\Big)\leq
\Eground\Big(\frac1b,\tau\Big)\,,$$
which can be re-written in the form,
\begin{equation}\label{eq:HP-ub*}
\limsup_{\min(K_1,K_2,K_3)\to\infty}\gse\leq
\sqrt{q\tau}\Eground\left(\frac1b,\tau\right)+o(\kappa )\,.
\end{equation}

\subsubsection{Lower bound}
It is proved in \cite{HP08} that,
\begin{equation}\label{En-eq-HP}
\lim_{\min(K_1,K_2,K_3)\to\infty}{\gse}=\inf_{(\psi,\nb_0)\in H^1(\Omega;\C)\times \mathcal C(\tau)}\mathcal G(\psi,\nb_0)\,,
\end{equation}
where $\mathcal C(\tau)$ is introduced in \eqref{C(tau)} and $\mathcal G$ is the functional in \eqref{eq:GL}.
Since the Oseen-Frank energy in \eqref{eq-OFE} vanishes for all $\nb_0\in\mathcal C(\tau)$, we get, for all $(\psi,\nb_0)\in H^1(\Omega;\C)\times \mathcal
C(\tau)$ and  $K_1,K_2,K_3\geq 0$,
$$\mathcal G(\psi,\nb_0)\geq \gse\,.
$$
In particular, when  $K_1=K_2=K_3=\kappa^4$ and $q\tau=b\kappa^2$, Assumptions~\ref{assumption:A} and \ref{assumption:A'} are satisfied and we can use Theorem~\ref{thm:en*} to write a lower bound for  $\gse$. That way we get,
$$\mathcal G(\psi,\nb_0)\geq \sqrt{q\tau}\,\Eground\big(\frac1b,\tau\big)+\kappa f(\kappa)\,,$$
where $f(\kappa)$ is a function satisfying
$f(\kappa)=o(1)$ as $\kappa\to\infty$.
Consequently, \eqref{En-eq-HP} yields,
 \begin{equation}\label{eq:HP-lb*}
\lim_{\min(K_1,K_2,K_3)\to\infty}{\gse}\geq \sqrt{q\tau}\Eground\big(\frac1b,\tau\big)+\kappa f(\kappa)\,.
\end{equation}
Combining the upper and lower bounds in \eqref{eq:HP-ub*} and \eqref{eq:HP-lb*} finishes the proof of Theorem~\ref{thm:gse}.\qed


\begin{thebibliography}{hhhhh}

\bibitem{A} Y. Almog. Thin boundary layers of chiral smectics. {\it Calc. Var. PDE.} {\bf 33} (2008), 299--328.


\bibitem{BCLP} P. Bauman, M. Calderer, C. Liu and D. Phillips. The phase transition between chiral nematic and smectic A$^*$ liquid crystals. {\it Arch. Rational Mech. Anal.} {\bf 165} (2002), 161--186.

\bibitem{C} M. C. Calderer, {\it Studies of layering and chirality
of smectic $A^*$ liquid crystals}, Math. Computer Modelling, {\bf
34} (2001), 1273-1288.


\bibitem{dGe} P.~G.~De\,Gennes.
{\it Superconductivity of metals and alloys.} Westview Press (1966)



\bibitem{deGe} P.~G.~De\,Gennes. An analogy between
superconductivity and smectics\,A. {\it Solid State Commun.} {\bf
10} (1972), 753--756.
%

\bibitem{dGP} P. G. de Gennes and J. Prost, {\it The Physics of
Liquid Crystals}, second edition, Oxford Science Publications,
Oxford, 1993.

\bibitem{DuGP} Q.~Du, M.D. Gunzburger and J.S.~Peterson. Analysis and approximation of the Ginzburg--Landau
 model of superconductivity. {\it SIAM Rev.} {\bf 34} (1) (1992), 54--81.

\bibitem{E} J.L. Ericksen. {\it Introduction to the Thermodynamics of Solids.}
Appl. Math. Sci., {\bf 131}, Springer-Verlag (1998).


\bibitem{FH-b} S. Fournais and B. Helffer. {\it Spectral Methods in Surface Superconductivity.}
Progress in Nonlinear Differential Equations and Their Applications,
{\bf 77} Birk\-h\"{a}user Boston (2010).

\bibitem{FK-cpde} S. Fournais, A. Kachmar. The ground state energy of the three
dimensional  Ginzburg-Landau functional. Part~I: Bulk regime.
{\it Comm. Partial Diff. Equations,} {\bf 38} (2013) 339-383.

\bibitem{FKP3D} S. Fournais, A. Kachmar and M.~Persson. The ground state energy of the three
dimensional  Ginzburg-Landau functional. Part~II: Surface Regime.
{\it J. Math. Pure Appl.} {\bf 99} (2013) 343-374.

\bibitem{HKL} R. Hardt, D. Kinderlehrer and F. H. Lin.
Existence and partial regularity of static liquid crystal
configurations. {\it Comm. Math. Phys.} {\bf 105} (1986), 547-570.


\bibitem{HK} B. Helffer and A. Kachmar. The Ginzburg-Landau functional with vanishing magnetic field. {\it Arch.  Rational Mech. Anal.} {\bf 218} (2015) 55-122.

\bibitem{HP08} B.~Helffer and X.-B.~Pan.
{Reduced Landau--de Gennes functional and surface smectic state of
liquid crystals.} {\it J. Funct. Anal.} {\bf 255} (2008),
3008--3069.

\bibitem{hemo} B. Helffer and A. Morame.
Magnetic bottles for the {N}eumann problem: curvature effects in the
case of dimension 3 (general case). {\it Ann. Sci. \'Ecole Norm.
Sup. (4)} {\bf 37} (2004), 105--170.

\bibitem{JP} S. Y. Joo and D. Phillips.  The phase transitions
from chiral nematic toward smectic liquid crystals. {\it Comm. Math.
Phys.}, {\bf 269} (2007), 367�C399.

\bibitem{L1} F.H. Lin. Nonlinear theory of defects in nematic liquid crystals; phase transition and flow phenomena. {\it Comm. Pure Appl. Math.} {\bf 42} (6) (1989), 789-814.

\bibitem{L2} F.H. Lin.  {\it Static and moving defects in liquid crystals.}  Proceedings of the International Congress of Mathematicians, Vol. I, II (Kyoto, 1990), 1165-1171, Math. Soc. Japan, Tokyo, 1991.

\bibitem{L3} F.H. Lin.  On nematic liquid crystals with variable degree of orientation. {\it Comm. Pure Appl. Math.} {\bf 44} (1991), no. 4, 453-468.

\bibitem{LL1} F.H. Lin and C. Liu.  Nonparabolic dissipative systems modeling the flow of liquid crystals. {\it Comm. Pure Appl. Math.} {\bf 48} (5) (1995), 501-537.

\bibitem{LL2} F.H. Lin and C. Liu. Existence of solutions for the Ericksen-Leslie system.  {\it Arch. Rational Mech. Anal.} {\bf 154} (2) (2000), 135-156.

\bibitem{LL3} F. H. Lin and C. Liu.  Static and dynamic
theories of liquid crystals. {\it J. Partial Diff. Equations,} {\bf 14}
(2001), 289-330.

\bibitem{LP} F.H. Lin and X.-B. Pan.  Magnetic field-induced instabilities in liquid crystals. {\it SIAM J. Math. Anal.} {\bf 38} (5) (2006/07), 1588-1612.


\bibitem{LW} F.H. Lin and C.Y. Wang.  Recent developments of analysis for hydrodynamic flow of nematic liquid crystals. {\it Philos. Trans. R. Soc. Lond. Ser. A Math. Phys. Eng. Sci.} {\bf 372} (2029) (2014),  20130361, 18 pp.


\bibitem{P} X.-B. Pan. Surface superconductivity in applied magnetic fields above
              {$H_{C_2}$}.  {\it Comm. Math. Phys.} {\bf 228} (2) (2002), 327--370.


\bibitem{P2} X.-B. Pan. Landau-de Gennes model of liquid crystals and
critical wave number.  {\it Comm. Math. Phys.} {\bf 239}
(1-2) (2003), 343--382.

\bibitem{Pa} X.-B. Pan. Surface superconductivity in 3 dimensions.
{\it Trans. Amer. Math. Soc.}  {\bf 356} (10) (2004), 3899--3937.

\bibitem{P4} X.-B. Pan.  Analogies between superconductors and
liquid crystals: nucleation and critical fields, in: {\it
Asymptotic Analysis and Singularities}, Advanced Studies in Pure
Mathematics, Mathematical Society of Japan, Tokyo,  {\bf 47-2}
(2007), 479--518.

\bibitem{R} N. Raymond. Contribution to the asymptotic analysis of the Landau-de Gennes functional.  {\it Advances in Diff. Equations} {\bf 15} (2010) 159-180.
\end{thebibliography}
\end{document}